\documentclass{imanum}

\usepackage{graphicx}
\usepackage{epstopdf}
\usepackage[caption=false,font=small]{subfig}
\usepackage{mathtools}
\usepackage{amscd}
\usepackage{stmaryrd}
\usepackage{mathrsfs}
\usepackage{ifpdf}
\usepackage{xcolor}
\usepackage{etoolbox}
\usepackage{tikz}
\usetikzlibrary{arrows.meta}
\usepackage{float}
\usepackage{placeins}
\usepackage{hyperref}

\ifpdf
  \DeclareGraphicsExtensions{.eps,.pdf,.png,.jpg}
\else
  \DeclareGraphicsExtensions{.eps}
\fi
\graphicspath{{figures/}}

\jno{}
\makeatletter
\let\@received\@empty
\makeatother

\title{$H^2$-based equilibrated residual error estimates}
\shorttitle{$H^2$ equilibrated residual error estimates}
\author{%
{\sc Kexin Ding\thanks{Email: dingkexin@zju.edu.cn}
and Yuwen Li\thanks{Email: liyuwen@zju.edu.cn}} \\[2pt]
School of Mathematical Sciences, Zhejiang University, 866 Yuhangtang Road, 
Hangzhou, 310058, P.R. China.
}
\shortauthorlist{K. Ding and Y. Li}

\ifpdf
\hypersetup{
  hidelinks,
  pdftitle={H2-based equilibrated residual error estimates},
  pdfauthor={Kexin Ding and Yuwen Li}
}
\fi

\makeatletter
\patchcmd{\ps@plain}{\end{tabular}}{\end{tabular}\hfil}{}{}
\patchcmd{\ps@plain}
  {\textit{IMA Journal of Numerical Analysis} (\number\year) }
  {}{}{}
\appto{\ps@plain}{\let\@oddfoot\@empty\let\@evenfoot\@empty}
\makeatother
\begin{document}

\maketitle

\begin{abstract}
{We develop a framework for equilibrated residual a posteriori error
analysis of finite element methods, based on decomposing vector- and tensor-valued Sobolev spaces into $H^2$ and $H^1$ spaces. To estimate the residual functional acting on $H^2$ component, the method reconstructs symmetric and piecewise polynomial
tensors on vertex patches in broken
$H(\operatorname{div}\operatorname{div})$ spaces. We apply the proposed framework
to conforming finite element discretizations of biharmonic equation, Poisson's equation, and linear elasticity in mixed form. As a result, we obtain several novel equilibrated residual a posteriori error estimates, providing guaranteed error control for energy error of $C^1$ method for biharmonic equation, $L^2$ norm error of Poisson's equation, and stress error of mixed linear elasticity. They also serve as the first polynomial-degree-robust a posteriori lower bounds of discretization errors. Numerical experiments illustrate the effectiveness of the $H^2$-based error
estimators.}
{a posteriori error estimate; equilibrated $H^{-2}$ residual; polynomial-degree robust; biharmonic equation;
Poisson equation; mixed linear elasticity.}
\end{abstract}

\section{Introduction}

Equilibrated residual a posteriori error estimation provides guaranteed error
bounds of finite element methods (FEMs) by locally reconstructing flux fields that satisfy  appropriate equilibrium
equations \citep{AinsworthOden2000,BraessSchoberl2008}.
For second-order elliptic problems, this typically amounts to lifting an
$H^{-1}$ residual into an $H(\operatorname{div})$ flux on local patches. For FEMs based on high-order polynomials, a guaranteed upper bound is most useful when the
estimator also remains efficient as the polynomial degree $p$ increases.
This requirement has motivated the development of polynomial-degree-robust ($p$-robust) equilibration
\citep{BraessPillweinSchoberl2009,ErnVohralik2015}.
A key analytical tool is a degree-independent comparison between continuous
and polynomial liftings of the same residual data. Stable broken $H^1$
and $H(\operatorname{div})$ polynomial extensions on vertex patches provide
such comparisons and underpin $p$-robust potential and flux reconstruction
\citep{ErnVohralik2020}.

For fourth-order problems, equilibrated bending-moment reconstructions
have been developed for HHJ and discontinuous Galerkin discretizations
\citep{BraessHoppeLinsenmann2018,BraessPechsteinSchoeberl2020}, with
extensions to goal-oriented estimation and nonlinear problems
\citep{Mallik2022,Hoppe2022}. For the mixed biharmonic problem,
a recent alternative auxiliary-space approach \citep{Li2025Auxiliary}, originating from \citet{LiZikatanov2025mcom}, uses classical $H^{-1}$
equilibration to derive $p$-robust a posteriori error estimate. Based on the auxiliary space viewpoint in \citet{Li2025Auxiliary,LiZikatanov2025mcom}, $p$-robust a posteriori error analysis of many important models, such as biharmonic equation in primal form and linear elasticity in mixed form, requires equilibrating a piecewise polynomial
residual in $H^{-2}$, which may contain element, edge, and vertex contributions. In this work, we develop such an equilibrated construction with $p$-robust efficiency for the first time in two dimensions, by representing each $H^{-2}$ residual as a
distributional double divergence of a symmetric tensor.

Our construction uses the polynomial spaces and trace structure associated
with the $\operatorname{div}\operatorname{div}$ complex from finite element exterior calculus
\citep{ArnoldHu2021,HuLiangMa2022,ChenHuang2025MCOM,HuLinZhang2025}.
A partition of unity localizes the residual, and a symmetric tensor of
minimum $L^2$ norm is computed on each vertex patch
subject to exact element, edge, and vertex equilibrium. The local tensors are then summed
to form a global reconstruction.
The central analytical result is a polynomial-degree-independent stability estimate
for local liftings satisfying all prescribed edge and vertex constraints
on the patch boundary.
Exact equilibration yields reliability, whereas the lifting and
localization bounds ensure degree-independent efficiency.

We apply the proposed $H^{-2}$-equilibration and auxiliary space framework to three model problems.
For the clamped biharmonic equation, the quintic Argyris $C^1$ partition of unity
gives a guaranteed and $p$-robust estimator of the $H^2$-seminorm error for
Argyris methods of degree $p\ge6$, which seems to be the first $p$-robust a posteriori error estimate for conforming discretizations of the biharmonic equation in primal form.
For Poisson's equation on a convex polygon, an analogous
reconstruction represents the residual on $H^2(\Omega)\cap H_0^1(\Omega)$
and yields the first guaranteed, $p$-robust $L^2$ norm error estimator for FEMs of polynomial degree $p\ge2$.
This construction uses a $C^0$ piecewise linear nodal partition of unity, which is much simpler than the $C^1$ Argyris partition.

For linear elasticity in mixed form, we consider the Hu--Zhang mixed symmetric tensor-valued finite element of stress polynomial degree
$p\ge4$ as an example. But our analysis should work for other conforming mixed elasticity finite elements with symmetric stress as well, such as the Arnold--Winther element \citep{AW2002}.
A compliance-weighted Helmholtz decomposition separates the stress error
into an Airy component, controlled by the unweighted $H^{-2}$
reconstruction, and a component bounded by data oscillation. As a result, we obtain a new error estimator simultaneously providing a guaranteed upper bound and a $p$-robust lower bound for the stress error. Prior to our work, \citet{DorsekMelenk2013} developed a $p$-robust error estimator for linear elasticity in primal formulation and  \citet{Kim2011} provided a guaranteed a posteriori estimator for mixed
methods with weak stress symmetry. In addition, a well-known robustness issue arises from linear elasticity, where analysis should remain uniform as the Lam\'e parameter $\lambda$ tends
to infinity \citep{CarstensenDolzmann1998,AdlerHuLiZikatanov2024}.
Such $\lambda$-uniform estimates have been obtained for symmetric mixed
methods \citep{ChenHuHuangMan2018} and for weakly symmetric stress
equilibration \citep{BertrandKoberMoldenhauerStarke2021}. Our $H^2$-based a posteriori error estimate is also uniform with respect to $\lambda$.

The paper is organized as follows.
Section~\ref{sect:prelim} introduces the residual representation and local
reconstruction of $H^{-2}$ residuals, with an application to biharmonic equations. Sections~\ref{sect:l2-error-estimate}--\ref{sect:stress} derive the equilibrated residual estimates for the $L^2$ norm error control of Poisson's equation and linear elasticity in mixed formulation.
Section~\ref{sect:numerical} presents numerical experiments.
Appendix~\ref{app:p-robust-lifting} contains the polynomial-degree-independent lifting
proofs, and Appendix~\ref{app:p-robust-localization} contains the localization
proofs.

\section{Preliminaries and Equilibration in \texorpdfstring{$H^{-2}$}{H-2}}\label{sect:prelim}
Throughout this paper, $d=2$. Let $\Omega\subset\mathbb R^2$ be a Lipschitz
polygonal domain, let $\mathbb V=\mathbb R^2$, and let $\mathbb S$ be the space
of $2\times2$ symmetric matrices.
We use ordinary italic letters for scalar, vector, and tensor fields;
their types will be specified by the associated function spaces.
By \mbox{$L^2(\Omega,\mathbb{V})=L^2(\Omega)\otimes\mathbb{V}$}
and \mbox{$L^2(\Omega,\mathbb{S})=L^2(\Omega)\otimes\mathbb{S}$} we denote the
space of \mbox{$\mathbb{V}$}- and \mbox{$\mathbb{S}$}-valued \mbox{$L^2$}
functions, respectively. Similar notation applies to function spaces with
different regularity as well as discrete spaces. By \mbox{$v_i$} and
\mbox{$\tau_i$} we denote the \mbox{$i$}-th entry of a vector \mbox{$v$}
and the \mbox{$i$}-th row of a matrix \mbox{$\tau$}, respectively.
Given
$w\in C^\infty(\Omega)$, $v\in C^\infty(\Omega,\mathbb{V})$,
$\tau\in C^\infty(\Omega,\mathbb{S})$,
define the following differential operators
\begin{align*}
&{\rm curl}w=\Big(\frac{\partial w}{\partial x_2},-\frac{\partial w}{\partial x_1}\Big),\quad{\rm rot}v=\frac{\partial v_2}{\partial x_1}-\frac{\partial v_1}{\partial x_2},\\
&\nabla v=\begin{pmatrix}
    \nabla v_1\\
    \nabla v_2
\end{pmatrix},\quad{\rm div}\tau=\begin{pmatrix}
    {\rm div}\tau_1\\ 
    {\rm div}\tau_2
\end{pmatrix},\\
&\varepsilon(v)={\rm symgrad}v=\frac{1}{2}\big(\nabla v+(\nabla v)^\top\big),\\
&\nabla^2w=\nabla(\nabla w).
\end{align*}

By $(\bullet,\bullet)_\Omega$ we denote the
$L^2(\Omega)$ inner product.
We write $\|\cdot\|_\Omega$ for the $L^2(\Omega)$ norm and
$|\cdot|_{k,\Omega}$ for Sobolev seminorms. Tensor-valued $L^2$
inner products use the Frobenius product
$\tau:\xi=\sum_{j,\ell=1}^2\tau_{j\ell}\xi_{j\ell}$.
All $H^2$ seminorms below use the Hessian Frobenius norm.
We omit the subscript $\Omega$ for global inner products
and norms.
Let $\langle\cdot,\cdot\rangle$ denote the duality pairing between a
Hilbert space $V$ and its dual
space $V^\prime$.
For a bounded linear operator $D:V\to W$ between Hilbert spaces,
$D^\prime:W^\prime\to V^\prime$ denotes its dual operator, defined by
$\langle D^\prime g,v\rangle=\langle g,Dv\rangle$
for $g\in W^\prime$ and $v\in V$.
We equip $H_0^2(\Omega)$ with the inner product
$(\nabla^2\cdot,\nabla^2\cdot)_\Omega$ and use the induced dual norm
\[
 \|R\|_{H^{-2}(\Omega)}
 :=\sup_{0\ne\omega\in H_0^2(\Omega)}
   \frac{|\langle R,\omega\rangle|}{|\omega|_{2,\Omega}}.
\]
Let $H^{-1}(\Omega)=H_0^1(\Omega)^\prime$ be the dual space of $H_0^1(\Omega)$.
The space $H_0^1(\Omega)$ is equipped with the inner product
$\langle-\Delta\bullet,\bullet\rangle=(\nabla\bullet,\nabla\bullet)$.
 
Let $\mathcal T_h$ be a conforming triangulation of $\Omega$
into triangles. The finite element space $V_h$ is defined
separately for each model problem in the corresponding section.
Let $\mathcal V_h:=\{a_i\}_{i\in\mathcal I}$ and $\mathcal E_h$
denote all vertices and edges of $\mathcal T_h$, respectively.
We denote by $\mathring{\mathcal V}_h$ and $\mathring{\mathcal E}_h$
the sets of interior vertices and interior edges, respectively.
For each grid vertex $a_i$, define the element patch $\mathcal T_i$
and its vertex patch $\Omega_i$ by
\[
\mathcal T_i:=\{T\in\mathcal T_h:a_i\in\overline T\},
\qquad
\Omega_i:=\operatorname{int}\!\left(\bigcup_{T\in\mathcal T_i}\overline T\right).
\]
We set $h_T:=\operatorname{diam}(T)$, $h_e:=\operatorname{diam}(e)$,  $h_i:=h_{a_i}:=\operatorname{diam}(\Omega_i)$.

Let $n$ be the piecewise unit vector field on
$\mathcal{E}_h$ with $n_e$ being the unit normal to
$e\in\mathcal{E}_h$. In $\mathbb{R}^2$,
$t:=n^\perp=(n_2,-n_1)$ is the piecewise unit tangent vector field
on $\mathcal{E}_h$. By $e=[a,b]\in\mathcal{E}_h$ we mean
$a, b$ are endpoints of $e$ and $t_e$ points from
$a$ to $b$. 
Given an interior edge $e$ shared by elements $T_+$, $T_-$,
let $\llbracket w\rrbracket_e=w|_{T_+}-w|_{T_-}$ denote the jump of a piecewise smooth function $w$ across the edge $e$, where the unit normal $n=n_e$ along $e$ is pointing from $T_+$ to $T_-$. When $e\subset\partial\Omega$, we set $n_e$ as the outward unit normal to $e$.

We interpret $D$ distributionally and write $D_h$ for its elementwise
counterpart on $\mathcal{T}_h$. For example, given a piecewise
smooth vector field $\sigma$, the distributional divergence
${\rm div}\sigma\in H^{-1}(\Omega)$ of $\sigma$ reads
\begin{equation*}
\langle{\rm div}\sigma,v\rangle=({\rm div}_h\sigma,v)-(\llbracket\sigma\cdot n\rrbracket,v)_{\mathcal{E}_h},\quad v\in H_0^1(\Omega).
\end{equation*}

\subsection{Equilibration in \texorpdfstring{$H^{-2}$}{H-2}}
\label{subsect:hminus2-equilibration}
On an edge $e\in\mathcal{E}_h$, there are two traces
\begin{equation*}
    \begin{aligned}
        &{\rm tr}_{1,e}\tau=\tau_{nn}=n_e\cdot\tau\cdot n_e,\\
        &{\rm tr}_{2,e}\tau
        =n_e\cdot{\rm div}\tau+\partial_t(\tau_{nt}).
    \end{aligned}
\end{equation*}

We define the vertex trace jump by
$$
\llbracket{\rm tr}_{3,a}\tau\rrbracket
=\sum_{e\in \mathcal{E}_h,a\in e}s_a^e\llbracket\tau_{nt}\rrbracket_e(a)
=\sum_{e\in \mathcal{E}_h,a\in e}
 s_a^e n_e^\top\llbracket\tau\rrbracket_e(a)t_e,
$$
where $s_a^e$ is the sign of a vertex relative to its incident edge:
$$s^e_a:=
\begin{cases}
1, & a\text{ is the end point of }e,\\
-1,& a\text{ is the start point of }e
\end{cases}.$$
On physical boundary edges, i.e., edges on $\partial\Omega$, the jump denotes the one-sided trace with the outward normal. For patchwise tensors, jumps are computed after zero extension outside the patch.

For a triangle $T$ with outward unit normal $n_T$ to $\partial T$,
we introduce the element vertex coefficient by
\begin{equation*}
 c_a(\tau):=\sum_{\substack{e\subset\partial T\\a\in e}}
 s_a^e\bigl(t_e\cdot\tau n_T\bigr)(a).
\end{equation*}
For each edge $e\subset\partial T$, set
$s_{T,e}:=n_T\cdot n_e\in\{-1,1\}$.
For a smooth symmetric tensor $\tau$ and $q\in H^2(T)$, integrating
twice by parts and then tangentially on each edge gives the Green identity
\citep[Lemma~2.1]{ChenHuang2025MCOM}
\begin{equation}\label{eq:element-green}
\begin{aligned}
 (\tau,\nabla^2q)_T
 &=\bigl(\operatorname{divdiv}\tau,q\bigr)_T+\sum_{a\in\operatorname{vert}(T)}c_a(\tau)q(a)\\
 &+\sum_{e\subset\partial T}
 \left[(\operatorname{tr}_{1,e}\tau,\partial_{n_T}q)_e
       -s_{T,e}(\operatorname{tr}_{2,e}\tau,q)_e\right].
\end{aligned}
\end{equation}
The element coefficients assemble into the vertex jump:
$\sum_{T\ni a}c_a(\tau|_T)=\llbracket\operatorname{tr}_{3,a}\tau\rrbracket$.
For the role of the edge and vertex jumps in
$H(\operatorname{div}\operatorname{div})$ conformity, see
\citet[Proposition~3.6]{FuehrerHeuerNiemi2019}.

Summing \eqref{eq:element-green} over the triangles gives
\begin{equation*}
\begin{aligned}
    \langle{\rm divdiv}\tau,v\rangle&=\sum_{T\in\mathcal{T}_h}({\rm divdiv}\tau,v)_T+\sum_{e\in\mathcal{E}_h}(\llbracket{\rm tr}_1\tau\rrbracket,\partial_nv)_e\\
    &-\sum_{e\in\mathcal{E}_h}(\llbracket{\rm tr}_2\tau\rrbracket, v)_e+\sum_{a\in\mathcal{V}_h}\llbracket{\rm tr}_{3,a}\tau\rrbracket v(a),\quad v\in H_0^2(\Omega).
\end{aligned}
\end{equation*}

A key step is to estimate the dual norm of an
$H^{-2}$-residual $R$. To this end, we seek
$\tilde{\sigma}_h\in L^2(\Omega,\mathbb S)$ such that
${\rm divdiv}\tilde{\sigma}_h=R$. By definition we have
\begin{equation*}
    \begin{aligned}
       \|R\|_{H^{-2}(\Omega)}&=\sup_{v\in H^2_0(\Omega),|v|_2=1}\langle R,v\rangle=\sup_{v\in H^2_0(\Omega),|v|_2=1}\langle{\rm divdiv}\tilde{\sigma}_h,v\rangle\\
       &=\sup_{v\in H^2_0(\Omega),|v|_2=1}(\tilde{\sigma}_h,\nabla^2v)\leq\|\tilde{\sigma}_h\|_{L^2(\Omega)}.
    \end{aligned}
\end{equation*}
This process is achieved via localization.

Unless otherwise stated, let $\{\phi_i\}_{i\in\mathcal I}$
be the degree-five Argyris vertex-value basis functions for all mesh
vertices. They form a $C^1$ partition of unity with
${\rm supp}(\phi_i)\subseteq\overline{\Omega_i}$ and satisfy
\begin{equation}\label{eq:argyris-partition-bound}
 \phi_i(a_j)=\delta_{ij},
 \qquad
 \sum_{j=0}^2h_T^j\|D^j\phi_i\|_{L^\infty(T)}
 \leq C(\gamma_i),\qquad T\in\mathcal T_i.
\end{equation}
For interior vertices $a_i$, $\phi_i\in H_0^2(\Omega)$.
For $R\in H^{-2}(\Omega)$, its localized
residual $R_i\in H^{-2}(\Omega)$ at $a_i$ is
\begin{equation*}
\langle R_i,v\rangle:=\langle R,v\phi_i\rangle\quad\forall v\in H_0^2(\Omega).    
\end{equation*} 

We shall consider the following typical residuals
\begin{align*}
    R&=f-\Delta^2u_h,\quad\text{from $C^1$ FEM for biharmonic equation},\\
    R&=f+\Delta u_h,\quad\text{from $C^0$ FEM for Poisson equation},\\
    R&=({\rm curlcurl})^\prime\mathcal{A}\sigma_h,\quad\text{from mixed method for linear elasticity}.
\end{align*}
The above residual $R$ has a piecewise polynomial
representation, i.e.,
\[
\langle R,v\rangle
=\sum_{T\in\mathcal{T}_h}({R_T},v)_T
+\sum_{e\in\mathcal{E}_h}\Big\{(R^n_{e},\partial_n v)_e
+(R^b_{e},v)_e\Big\}+\sum_{a\in\mathcal{V}_h}R^g_{a}v(a),
\]
for $v\in H_0^2(\Omega)$.
Here $R_T$ is a polynomial on $T$, $R_e^n$ and $R_e^b$ are polynomials
on $e$, and $R_a^g\in\mathbb R$ is the coefficient at the vertex $a$.
For convenience, denote these coefficient families by
$R^o:=\{R_T\}$, $R^b:=\{R_e^b\}$, $R^n:=\{R_e^n\}$, and
$R^g:=\{R_a^g\}$; use $R_i^o,R_i^b,R_i^n,R_i^g$ for their localized counterparts.

As illustrated in Figure~\ref{fig:patch-equilibration}, the local
equilibrium condition ${\rm divdiv}\tau=R_i$ means
\begin{align*}
    {\rm divdiv}_h\tau&=R_i^o=R^o\phi_i,\\
    \llbracket{\rm tr}_1\tau\rrbracket&=R^n_i=R^n\phi_i,\quad \llbracket{\rm tr}_{3,a}\tau\rrbracket=R_i^g=R^g\phi_i(a),\\
    -\llbracket{\rm tr}_2\tau\rrbracket&=R^b_i=R^b\phi_i+R^n\partial_n\phi_i.
\end{align*}

For a boundary patch, set
$\Gamma_i^c:=\partial\Omega_i\setminus\partial\Omega$; for an interior patch,
set $\Gamma_i^c:=\partial\Omega_i$.
Set $\Gamma_i^f:=\partial\Omega_i\cap\partial\Omega$, and let
$\mathcal V_i^c$ be the mesh vertices contained
in $\Gamma_i^c$.
For a reconstruction degree $k\geq3$, the local space for broken
$H({\rm divdiv})$ flux reconstruction is
\begin{align*}
\Sigma_{h,k}(\Omega_i):=&\{\tau\in L^2(\Omega,\mathbb S):
 {\rm supp}(\tau)\subseteq\overline{\Omega_i},\\
&\tau|_T\in\mathcal P_k(T,\mathbb S)~\forall T\in\mathcal T_i,\\
&{\rm tr}_1\tau={\rm tr}_2\tau=0\text{ on }\Gamma_i^c,\quad
 \llbracket{\rm tr}_{3,a}\tau\rrbracket=0~\text{for}~a\in\mathcal V_i^c\}.
\end{align*}
\begin{figure}[tbp]
\centering
\begin{tikzpicture}[x=1cm,y=1cm,
  every node/.style={font=\fontsize{7}{8.5}\selectfont,inner sep=1pt},
  patchcell/.style={draw=black!70,line width=0.55pt,fill=black!2},
  firsttrace/.style={draw=cyan!75!blue,line width=0.45pt,double distance=1pt},
  interiortrace/.style={draw=cyan!75!blue,line width=0.55pt}]
\coordinate (L) at (1.5,4.35);
\coordinate (R) at (7.6,3.8);
% Keep the five tensor pieces and their numbering as in the source sketch.
\foreach \j/\ang/\nextang/\mid/\celllabeldx/\celllabeldy in
  {2/126/54/90/0/0,3/54/-18/18/-0.10/-0.14,4/-18/-90/-54/0/0,
   5/-90/-162/-126/0/0,1/-162/-234/-198/0.10/-0.14}{
  \begin{scope}[shift={(L)},shift={(\mid:0.035)}]
    \path[patchcell] (0,0)--(\ang:1.22)--(\nextang:1.22)--cycle;
    \node[font=\small] at (\mid:0.68) {$\sigma_i^{\j}$};
  \end{scope}
  \begin{scope}[shift={(R)},shift={(\mid:0.78)}]
    \path[patchcell] (0,0)--(\ang:2.2)--(\nextang:2.2)--cycle;
    \node[align=center,xshift=\celllabeldx cm,yshift=\celllabeldy cm]
      at (\mid:1.20)
      {$\operatorname{divdiv}\sigma_i^{\j}$\\$=R_i^o|_{T_{\j}}$};
  \end{scope}
}
\fill[black!70] (L) circle[radius=1.3pt];
\node[below right=2pt,font=\small] at (L) {$a_i$};
\draw[draw=black!70,line width=0.8pt,-{Stealth[length=2.3mm]}]
  (2.95,4.10) to[bend left=6]
  node[above=3pt,font=\small] {$\operatorname{divdiv}$} (4.05,4.20);
% Show both trace jumps on every interior radial edge.
\foreach \ang/\rot/\labelradius in
  {126/-54/1.88,54/54/1.88,-18/-18/1.72,-90/90/1.72,-162/18/1.72}{
  \pgfmathsetmacro{\normalang}{\ang+90}
  \begin{scope}[shift={(R)},shift={(\normalang:0.40)}]
    \draw[interiortrace] (\ang:0.72)--(\ang:2.67);
  \end{scope}
  \begin{scope}[shift={(R)},shift={(\normalang:-0.40)}]
    \draw[interiortrace] (\ang:0.72)--(\ang:2.67);
  \end{scope}
  \begin{scope}[shift={(R)},shift={(\normalang:0.18)}]
    \node[rotate=\rot] at (\ang:\labelradius)
      {$\llbracket\operatorname{tr}_1\sigma_i\rrbracket=R_i^n$};
  \end{scope}
  \begin{scope}[shift={(R)},shift={(\normalang:-0.18)}]
    \node[rotate=\rot] at (\ang:\labelradius)
      {$-\llbracket\operatorname{tr}_2\sigma_i\rrbracket=R_i^b$};
  \end{scope}
}
% On the outer boundary of this interior patch both traces vanish.
\foreach \mid/\rot in {90/0,18/-72,-54/36,-126/-36,-198/72}{
  \begin{scope}[shift={(R)},shift={(\mid:2.95)}]
    \draw[firsttrace] (\rot:-1.02)--(\rot:1.02);
  \end{scope}
  \begin{scope}[shift={(R)}]
    \node[rotate=\rot] at (\mid:2.75) {$\operatorname{tr}_2\sigma_i=0$};
    \node[rotate=\rot] at (\mid:3.14) {$\operatorname{tr}_1\sigma_i=0$};
  \end{scope}
}
% Outer vertex symbols are offset outward from their separated copies.
\pgfmathsetmacro{\vertexradius}{2.2+0.78*cos(36)+0.20}
\foreach \j/\ang in {1/126,2/54,3/-18,4/-90,5/-162}{
  \coordinate (B\j) at ([shift={(\ang:\vertexradius)}]R);
  \fill[red] (B\j) circle[radius=2.3pt];
}
\fill[red] (R) circle[radius=2.3pt];
\node[align=center,anchor=south,font=\fontsize{7}{8.5}\selectfont] at ([shift={(0,0.08)}]R)
  {$\llbracket\operatorname{tr}_{3,a_i}\sigma_i\rrbracket$\\$=R_i^g$};
\node[align=center,anchor=south] at ([shift={(-0.17,0.19)}]B1)
  {$\llbracket\operatorname{tr}_{3,a}\sigma_i\rrbracket=0$};
\node[align=center,anchor=south] at ([shift={(0.17,0.19)}]B2)
  {$\llbracket\operatorname{tr}_{3,a}\sigma_i\rrbracket=0$};
\node[align=center,anchor=west] at ([shift={(0.20,-0.12)}]B3)
  {$\llbracket\operatorname{tr}_{3,a}\sigma_i\rrbracket=0$};
\node[align=center,anchor=north] at ([shift={(0,-0.18)}]B4)
  {$\llbracket\operatorname{tr}_{3,a}\sigma_i\rrbracket=0$};
\node[align=center,anchor=east] at ([shift={(-0.20,-0.12)}]B5)
  {$\llbracket\operatorname{tr}_{3,a}\sigma_i\rrbracket=0$};
\end{tikzpicture}
\caption{Local equilibration on an interior vertex patch:
tensor pieces $\sigma_i^j:=\sigma_i|_{T_j}$ (left) and their element,
edge, and vertex residual data (right).}
% Alt text: An interior patch of five triangles sharing vertex $a_i$. The left panel shows the tensor pieces $\sigma_i^j$. The right panel separates the triangles to display element residuals, edge trace jumps, and vertex contributions, with zero traces and vertex jumps on the outer boundary.
\label{fig:patch-equilibration}
\end{figure}

Following the vertex-patch equilibration strategy
of \citet{BraessPillweinSchoberl2009}, we compute a minimum-norm lifting
of $R_i$ on each patch.
\begin{problem}[local flux reconstruction]\label{prob:min_sigmadivdiv}
\begin{align*}
&\sigma_i=\arg\min_{\tau\in\Sigma_{h,k}(\Omega_i)}\|\tau\|_{L^2(\Omega_i)},\\
&\text{ subject to }{\rm divdiv}\tau=R_i\text{ in }H^{-2}(\Omega).
\end{align*}
\end{problem}

Let $U_i$ be the broken polynomial space of degree
$\leq k-2$ on $\Omega_i$, let $\Lambda^n_i$ and
$\Lambda^b_i$ be the broken polynomial spaces of degree
$\leq k$ and degree
$\leq k-1$ on
$\mathcal{E}_i$, respectively. Let $\Lambda_i^g$ be the space of scalar
values at the vertices in $\mathcal V_i$. Let $\mathcal{E}_i$ be the set of interior edges
in $\Omega_i$, and $\mathcal{V}_i$ the set of interior
vertices in $\Omega_i$.
Local residual matching is imposed on $\mathcal T_i$, $\mathcal E_i$, and
$\mathcal V_i$; boundary traces are prescribed through the reconstruction
space.
Set
\begin{equation*}
 U_{h,k}(\Omega_i):=
 U_i\times\Lambda_i^n\times\Lambda_i^b\times\Lambda_i^g.
\end{equation*}
For $q\in\mathcal P_1(\Omega_i)$, define its multiplier-tuple embedding by
\begin{equation*}
 \iota_iq
 :=\bigl(q,(\partial_nq)|_{\mathcal E_i},q|_{\mathcal E_i},
          (q(a))_{a\in\mathcal V_i}\bigr)\in
 U_{h,k}(\Omega_i),
\end{equation*}
and set
\begin{equation*}
 \mathcal N_i:=
 \begin{cases}
  \iota_i\mathcal P_1(\Omega_i),&a_i\in\mathring\Omega,\\
  \{0\},&a_i\in\partial\Omega.
 \end{cases}
\end{equation*}

The Euler-Lagrange equation of
Problem \mbox{\ref{prob:min_sigmadivdiv}} seeks
\begin{equation*}
 (\sigma_i,[u_i])\in
 \Sigma_{h,k}(\Omega_i)
 \times\bigl(U_{h,k}(\Omega_i)/\mathcal N_i\bigr).
\end{equation*}
Here $u_i=(u_i^o,u_i^n,u_i^b,u_i^g)$ denotes any representative of
$[u_i]$.
The equations are
\begin{subequations}\label{eq:sigmaiR}
\begin{align}
&(\sigma_i,\tau_i)_{\Omega_i}
+b_i(\tau_i,u_i)=0,
\quad \tau_i\in \Sigma_{h,k}(\Omega_i),\\
&b_i(\sigma_i,v_i)=(R_i^o,v_i^o)_{\Omega_i}+(R_i^n,v_i^n)_{\mathcal{E}_i}+(R_i^b,v_i^b)_{\mathcal{E}_i}+(R_i^g,v_i^g)_{\mathcal{V}_i},
\end{align}
\end{subequations}
for $v_i\in U_{h,k}(\Omega_i)$,
where the bilinear form $b_i$ is defined by
\begin{align*}
b_i(\tau_i,v_i)&=({\rm divdiv}_h
\tau_i,v_i^o)_{\Omega_i}
+(\llbracket{\rm tr}_1(\tau_i)\rrbracket,v^n_i)_{\mathcal{E}_i}\\
&
-(\llbracket{\rm tr}_2(\tau_i)\rrbracket,v^b_i)_{\mathcal{E}_i}+(\llbracket{\rm tr}_{3,a}(\tau_i)\rrbracket, v^g_i)_{\mathcal{V}_i}
\end{align*}
For an interior patch,
\begin{equation*}
 b_i(\tau_i,\iota_iq)
 =(\tau_i,\nabla^2q)_{\Omega_i}=0
 ,\qquad q\in\mathcal P_1(\Omega_i).
\end{equation*}
Thus the equations are independent of the multiplier representative.
The minimizing stress is unique whenever the constraint set is nonempty.

For $d=2$, the elementwise stress and multiplier spaces above
coincide, up to their ordering and diagonal mesh scalings, with
$\Sigma^{-1}_{k,k-2}$ and
$M^{-1}_{k-2,k-1,k,k}$ in
\citet[Section~4.1]{ChenHuang2025MCOM}.

Let $\sigma_i\in L^2(\Omega_i,\mathbb S)$, extended by zero, satisfy
$\operatorname{divdiv}\sigma_i=R_i$ in $H^{-2}(\Omega)$.
The partition of unity gives $\sum_iR_i=R$. Thus, for
$v\in H_0^2(\Omega)$,
$\langle R,v\rangle=(\sum_i\sigma_i,\nabla^2v)$, and
Cauchy--Schwarz yields
\begin{equation*}
 \operatorname{divdiv}\sum_i\sigma_i=R,\qquad
 \|R\|_{H^{-2}(\Omega)}\leq\Bigl\|\sum_i\sigma_i\Bigr\|.
\end{equation*}
For vertex patches, the overlap number
$N_{\rm ov}:=\mathop{\rm ess\,sup}_{x\in\Omega}
\#\{i:x\in\Omega_i\}$ is at most three. Consequently,
\begin{equation*}
 \Bigl\|\sum_i\sigma_i\Bigr\|^2
 \leq N_{\rm ov}\sum_i\|\sigma_i\|_{0,\Omega_i}^2.
\end{equation*}

We now derive polynomial-degree-independent reconstruction estimates
for the fixed degree-five Argyris partition
$\{\phi_i\}_{i\in\mathcal I}$.
Let $\rho_T$ be the inradius of $T$ and set
\begin{equation*}
 \gamma_T:=\frac{h_T}{\rho_T},\qquad
 \gamma_i:=\max_{T\in\mathcal T_i}\gamma_T,\qquad
 \gamma:=\max_{i\in\mathcal I}\gamma_i.
\end{equation*}

\begin{theorem}
\label{thm:p-robust-equilibration}
Let $k\ge3$ and let $R\in H^{-2}(\Omega)$ have a polynomial residual
representation with localized coefficient tuples
$(R_i^o,R_i^n,R_i^b,R_i^g)\in U_{h,k}(\Omega_i)$ for every $i\in\mathcal I$. If
$\langle R,\phi_i\ell\rangle=0$ for every affine $\ell$ on each interior patch,
the minimum-$L^2$ liftings $\sigma_i\in\Sigma_{h,k}(\Omega_i)$ of $R_i$
satisfy the following estimates.
\begin{equation*}
 \|R\|_{H^{-2}(\Omega)}
 \leq\Bigl\|\sum_i\sigma_i\Bigr\|_{0,\Omega}
 \leq\sqrt{3}\Bigl(\sum_i\|\sigma_i\|_{0,\Omega_i}^2\Bigr)^{1/2}
 \leq C(\gamma)\|R\|_{H^{-2}(\Omega)}.
\end{equation*}
The constant $C(\gamma)$ depends only on $\gamma$, and is independent
of the mesh size and the polynomial degrees.
\end{theorem}
\begin{proof}
For the local test spaces $X_i$ in
Appendix~\ref{app:p-robust-lifting},
Theorem~\ref{thm:p-robust-local-lifting} ensures local feasibility
and, by minimality, gives
$\|\sigma_i\|_{0,\Omega_i}\leq C(\gamma)\|R_i\|_{X_i'}$.
Lemma~\ref{lem:p-robust-localization} bounds the square sum of these
dual norms by $C(\gamma)\|R\|_{H^{-2}(\Omega)}^2$.
Finite overlap, with $N_{\rm ov}\leq3$, proves both upper bounds after
enlarging $C(\gamma)$.
The lower bound follows from the assembly identity above
and Cauchy--Schwarz.
\end{proof}

\subsection{Biharmonic Equation}\label{sect:natural}
The construction below applies the $H^{-2}$ equilibration framework of
Section~\ref{sect:prelim} directly to a conforming discretization of the
biharmonic equation.
We consider the clamped biharmonic problem
\begin{equation*}
\begin{aligned}
    \Delta^2u&=f\quad\text{ in }\Omega,\\
    u=\partial_nu&=0\quad\text{ on }\partial\Omega,
\end{aligned}
\end{equation*}
and its finite element discretization:
$u_h\in V_h\subset H_0^2(\Omega)$ satisfying
\begin{equation*}
    (\nabla^2u_h,\nabla^2v_h)=(f,v_h),\quad v_h\in V_h. 
\end{equation*}
Here $V_h$ is the Argyris finite element space of polynomial degree $p\ge6$. We assume $f|_T\in \mathcal P_{p-4}(T)$ for every $T\in\mathcal{T}_h$.
With $R:=f-\Delta^2u_h\in H^{-2}(\Omega)$, the Riesz representation identity gives:
\begin{equation*}
    |u-u_h|_{H^2(\Omega)}=\|R\|_{H^{-2}(\Omega)}.
\end{equation*}
For $v\in H_0^2(\Omega)$, elementwise integration by parts gives
\begin{equation*}
    \begin{aligned}
        \langle R,v\rangle&=(f,v)-(\nabla^2u_h,\nabla^2v)\\
        &=(f,v)-((\nabla^2u_h)n,\nabla v)_{\partial\mathcal{T}_h}+({\rm div}_h\nabla^2u_h,\nabla v)\\
        &=(f,v)-(\llbracket n\cdot\nabla^2u_h\rrbracket,\nabla v)_{\mathcal{E}_h}+({\rm div}_h\nabla^2u_h,\nabla v)\\
        &=(f-\Delta_h^2u_h,v)+(\llbracket n\cdot{\rm div}_h\nabla^2u_h\rrbracket,v)_{\mathcal{E}_h}-(\llbracket (\nabla^2u_h)_{nn}\rrbracket,\partial_n v)_{\mathcal{E}_h}.
    \end{aligned}
\end{equation*}
The last equality holds because
$n\cdot\nabla^2u_h\cdot t$ is continuous.

For the Argyris approximation of polynomial degree
$p$, we take the $H({\rm divdiv})$ element of reconstruction degree
$k=p+3$.
The residual representation and the localization rule in Section~\ref{sect:prelim}
give
$R_i^o=\phi_i(f-\Delta_h^2u_h)$,
$R_i^n=-\phi_i\llbracket(\nabla^2u_h)_{nn}\rrbracket$,
$R_i^b=\phi_i\llbracket n\cdot{\rm div}_h\nabla^2u_h\rrbracket
-(\partial_n\phi_i)\llbracket(\nabla^2u_h)_{nn}\rrbracket$,
and $R_i^g=0$.
With the localized residual coefficients specified above, we compute
the minimum-$L^2$ reconstruction by \eqref{eq:sigmaiR}, using $k=p+3$.
Set $\eta_h:=\|\sum_i\sigma_i\|_{0,\Omega}$ and
$\eta_T:=\|\sum_i\sigma_i\|_{0,T}$, so that
$\eta_h^2=\sum_{T\in\mathcal T_h}\eta_T^2$.
For local efficiency, set $\eta_i:=\|\sigma_i\|_{0,\Omega_i}$.
This patch indicator uses the single reconstruction $\sigma_i$;
$\eta_T$ is the element contribution to the assembled global estimator.

\begin{corollary}\label{cor:biharmonic-local-efficiency}
Let $u_h$ be the degree-$p$ Argyris FEM solution of the clamped
biharmonic equation, with $p\ge6$ and
$f|_T\in\mathcal P_{p-4}(T)$ for every $T\in\mathcal T_h$.
For reconstruction degree $k=p+3$, the equilibrated residual error
estimator  satisfies
\begin{align*}
 |u-u_h|_{2,\Omega}&\leq\eta_h\leq C(\gamma)|u-u_h|_{2,\Omega},\\
 \eta_i&\leq C(\gamma_i)|u-u_h|_{2,\Omega_i}.
\end{align*}
The constants are independent of $h$ and $p$.
\end{corollary}
\begin{proof}
The localized element, normal-derivative, and value coefficients have
degrees at most $p+1=k-2$, $p+3=k$, and $p+2=k-1$, respectively,
and $R_i^g=0$. Thus their tuples belong to $U_{h,k}(\Omega_i)$.
For every interior patch and $q\in\mathcal P_1(\Omega_i)$,
the zero extension of $\phi_iq$ belongs to $V_h$ because its degree
is at most six and $p\ge6$. Galerkin orthogonality therefore gives
$\langle R,\phi_iq\rangle=0$.
Theorem~\ref{thm:p-robust-local-lifting} ensures local feasibility;
no affine moment condition is required on a boundary patch.
The global bounds follow from Theorem~\ref{thm:p-robust-equilibration}
and the identity $\|R\|_{H^{-2}(\Omega)}=|u-u_h|_{2,\Omega}$.
For the patchwise bound, Theorem~\ref{thm:p-robust-local-lifting} and
Lemma~\ref{lem:p-robust-localization} with $E=u-u_h$ give
\begin{equation*}
 \eta_i=\|\sigma_i\|_{0,\Omega_i}
 \leq C(\gamma_i)\|R_i\|_{X_i'}
 \leq C(\gamma_i)|u-u_h|_{2,\Omega_i}.
\end{equation*}
This proves the assertion.
\end{proof}

\section{L2 Error Estimate of Poisson Problem}\label{sect:l2-error-estimate}
Consider a posteriori error estimate of
$\|u-u_h\|_{L^2(\Omega)}$ of the Poisson equation and its finite
element discretization:
$u_h\in V_h\subset H_0^1(\Omega)$ satisfying
\begin{equation*}
    (\nabla u_h,\nabla v_h)=(f,v_h),\quad v_h\in V_h. 
\end{equation*}
Here $V_h$ is the standard Lagrange finite element space
of degree $p\ge2$.
In this section,  we assume $f|_T\in \mathcal P_{p-2}(T)$ and set $a(w,v):=(\nabla w,\nabla v)$ for all
$w,v\in H_0^1(\Omega)$.
The residual is $R=f+\Delta u_h$,
$$\langle R,v\rangle=(f,v)-a(u_h,v),\quad v\in H_0^1(\Omega).$$
We consider the space
$\widetilde{H}^2(\Omega)=H^2(\Omega)\cap H_0^1(\Omega)$ and the dual norm
\[
\|R\|_{\widetilde{H}^2(\Omega)^\prime}=\sup_{v\in \widetilde{H}^2(\Omega),\, |v|_2=1}\langle R,v\rangle.
\]

Assume that $\Omega$ is convex.
To bound the $L^2$ error, let
$\varphi\in\widetilde{H}^2(\Omega)$
\space solve the dual problem
\begin{equation*}
    a(v,\varphi)=(u-u_h,v),\quad v\in H_0^1(\Omega).
\end{equation*}
The elliptic regularity gives
$|\varphi|_{H^2(\Omega)}\leq\|u-u_h\|_{L^2(\Omega)}$
(Miranda-Talenti inequality).
Then we have
\begin{equation*}
    \|u-u_h\|^2=a(u-u_h,\varphi)\leq\|R\|_{\widetilde{H}^2(\Omega)^\prime}|\varphi|_{H^2(\Omega)}
\end{equation*}
and thus $\|u-u_h\|\leq\|R\|_{\widetilde{H}^2(\Omega)^\prime}$.
For the lower bound,
\begin{align*}
\langle R,v\rangle=a(u-u_h,v)=(u-u_h,-\Delta v)\leq C\|u-u_h\|_{L^2(\Omega)}|v|_{H^2(\Omega)}.
\end{align*}

For $v\in\widetilde{H}^2(\Omega)$,
\begin{equation*}
\begin{aligned}
    \langle{\rm divdiv}\tau,v\rangle
    &=(\tau,\nabla^2v)\\
    &=({\rm divdiv}_h\tau,v)
                              +(\llbracket{\rm tr}_1\tau\rrbracket,\partial_nv)_{\mathring{\mathcal{E}}_h}
      {}+(\operatorname{tr}_1\tau,\partial_nv)_{\partial\Omega}\\
    &\qquad-(\llbracket{\rm tr}_2\tau\rrbracket,v)_{\mathring{\mathcal{E}}_h}
      +(\llbracket{\rm tr}_{3,a}\tau\rrbracket,v)_{\mathring{\mathcal{V}}_h}.
\end{aligned}
\end{equation*}
We seek $\sigma_h\in L^2(\Omega,\mathbb S)$ such that
\[
\langle R,v\rangle=(\sigma_h,\nabla^2v),\quad\forall v\in \widetilde{H}^2(\Omega).
\]
Then we have
\begin{equation*}
\|R\|_{\widetilde{H}^2(\Omega)^\prime}=\sup_{v\in \widetilde{H}^2(\Omega), |v|_{2}=1}(\sigma_h,\nabla^2v)\leq\|\sigma_h\|_{L^2(\Omega)}.
\end{equation*}
Elementwise integration by parts gives the residual representation 
\begin{equation*}
    \begin{aligned}
        &\langle R,v\rangle=(f,v)-(\nabla u_h,\nabla v)\\
        &=(f,v)-(\nabla u_h\cdot n,v)_{\partial\mathcal{T}_h}+({\rm div}_h\nabla u_h, v)\\
        &=(f+\Delta_hu_h,v)-(\llbracket \partial_nu_h\rrbracket,v)_{\mathring{\mathcal{E}}_h}
    \end{aligned}
\end{equation*}

In this section, we instead use the $C^0$ piecewise-linear
nodal partition of unity $\{\phi_i\}_{i\in\mathcal I}$. Since
$R\in H^{-1}(\Omega)$ and $\phi_i v\in H_0^1(\Omega)$ for every
$v\in\widetilde H^2(\Omega)$, this choice is sufficient for localization.
Split
$\langle R,v\rangle=(\sigma,\nabla^2v)$ into
$\langle R,v\phi_i\rangle=(\sigma_i,\nabla^2v)$ with
${\rm supp}\sigma_i\subseteq\overline{\Omega_i}$.

We use the $H({\rm divdiv})$ element of reconstruction degree $k=p+1$.
Since $\partial_n v$ need not vanish on the physical boundary, use
\begin{equation*}
\Sigma_{h,k}^{L^2}(\Omega_i)
:=\bigl\{\tau_h\in\Sigma_{h,k}(\Omega_i):
\operatorname{tr}_1\tau_h=0\text{ on }\Gamma_i^f\bigr\}
\end{equation*}
for both the stress trial and test spaces.
Then
\[
 R_i=(\phi_i(f+\Delta_hu_h),\,0,\,
                      -\phi_i\llbracket\partial_nu_h\rrbracket,\,0).
\]
With $f|_T\in\mathcal P_{p-2}(T)$,
\[
 R_i^o|_T\in\mathcal P_{p-1}(T),\qquad
 R_i^b|_e\in\mathcal P_p(e).
\]
Set
\[
 \mathcal K_i:=
 \begin{cases}
  \mathcal P_1(\Omega_i),&a_i\in\mathring\Omega,\\
  \{q\in\mathcal P_1(\Omega_i):q|_{\Gamma_i^f}=0\},
       &a_i\in\partial\Omega.
 \end{cases}
\]
The multiplier is understood modulo
$\mathcal N_i^{L^2}:=\iota_i\mathcal K_i$.
We solve \eqref{eq:sigmaiR} with $k=p+1$, replacing the stress space
by $\Sigma_{h,k}^{L^2}(\Omega_i)$ and the multiplier kernel by
$\iota_i\mathcal K_i$, as defined above.
Let $Y_i$ be the $H^2$ patch test space with zero values on the
physical radial edges, modulo its affine kernel $\mathcal L_i$,
as defined in Appendix~\ref{app:p-robust-lifting}.
For $\ell\in\mathcal L_i$, the zero extension of $\phi_i\ell$
belongs to $V_h$: its degree is at most two, $\phi_i$ vanishes
on outer edges, and $\ell$ vanishes on physical radial edges.
Galerkin orthogonality gives
\begin{equation*}
 \langle R_i,\ell\rangle_{\Omega_i}
 =\langle R,\phi_i\ell\rangle=0
 ,\qquad\ell\in\mathcal L_i.
\end{equation*}
Together with the degree bounds above, this verifies the hypotheses
of Theorem~\ref{thm:poisson-p-lifting} with $k=p+1$.
Thus every local feasible set is nonempty, and its minimum-$L^2$
solution satisfies
\begin{equation*}
 \|\sigma_i\|_{0,\Omega_i}
 \leq C(\gamma,\Omega)\|R_i\|_{Y_i'}.
\end{equation*}
Summing the local equilibrium identities and using the partition
of unity gives
\[
 {\rm divdiv}\,\sum\nolimits_{i\in\mathcal I}\sigma_i
 =R\quad\text{in }\widetilde H^2(\Omega)',
 \quad
 \|u-u_h\|_{0,\Omega}
 \leq\|R\|_{\widetilde H^2(\Omega)'}
 \leq\Bigl\|\sum\nolimits_{i\in\mathcal I}\sigma_i\Bigr\|_{0,\Omega}.
\]

Set $\eta_i:=\|\sigma_i\|_{0,\Omega_i}$,
$\eta_T:=\|\sum_i\sigma_i\|_{0,T}$, and
$\eta_h:=\|\sum_i\sigma_i\|_{0,\Omega}$, so that
$\eta_h^2=\sum_T\eta_T^2$.
\begin{theorem}
\label{thm:poisson-p-estimator}
Let $\Omega$ be a fixed convex polygon and let the conforming
triangulations be shape regular with parameter $\gamma$.
For the degree-$p$ Lagrange method and the minimum-$L^2$
reconstruction above, assume $p\geq2$ and $f|_T\in\mathcal P_{p-2}(T)$. Then
\begin{align*}
 \eta_i&\leq C(\gamma_i,\Omega)\|u-u_h\|_{0,\Omega_i},
 \qquad i\in\mathcal I,\\
 \|u-u_h\|_{0,\Omega}&\leq\eta_h
 \leq C(\gamma,\Omega)\|u-u_h\|_{0,\Omega}.
\end{align*}
The constants are independent of the mesh size and $p$.
\end{theorem}
\begin{proof}
By minimality, Theorems~\ref{thm:poisson-p-lifting} and
\ref{thm:poisson-p-localization} give
\begin{equation*}
 \eta_i\leq C(\gamma_i,\Omega)\|R_i\|_{Y_i'}
 \leq C(\gamma_i,\Omega)\|u-u_h\|_{0,\Omega_i}.
\end{equation*}
Finite overlap then yields
\begin{equation*}
 \eta_h^2\leq3\sum_i\eta_i^2
 \leq C(\gamma,\Omega)\|u-u_h\|_{0,\Omega}^2.
\end{equation*}
The reliability bound was proved above.
\end{proof}

% Earlier commented elasticity draft: archive/elasticity_before_constant_parameters.tex
\section{Linear Elasticity in Mixed Form}\label{sect:stress}
In this section, we derive a guaranteed, $\lambda$-robust estimate of
 the compliance-weighted stress error on a bounded contractible
Lipschitz polygon $\Omega$, with positive constant Lam\'e parameters
$\lambda$ and $\mu$. The estimates below are uniform in both parameters.
We consider the homogeneous
traction problem \citep{Li2021M2AN}
\begin{equation}\label{eq:elasticity}
\begin{aligned}
 \mathcal A\sigma&=\varepsilon(u)&&\text{in }\Omega,\\
 \operatorname{div}\sigma&=-f&&\text{in }\Omega,\\
 \sigma n&=0&&\text{on }\partial\Omega,
\end{aligned}
\end{equation}
where the compliance and elasticity tensors are
\[
 \mathcal A\tau=\frac1{2\mu}
 \left(\tau-\frac{\lambda}{2(\lambda+\mu)}(\operatorname{tr}\tau)I\right),
 \qquad
 \mathcal C\tau:=\mathcal A^{-1}\tau
 =2\mu\tau+\lambda(\operatorname{tr}\tau)I.
\]

The rigid-motion space is the kernel of the symmetric gradient:
\[
 {\rm RM}(\Omega)
 :=\bigl\{r\in H^1(\Omega,\mathbb V):\varepsilon(r)=0\bigr\}.
\]
We assume throughout that $f\in L^2(\Omega,\mathbb V)$ and
$f\perp{\rm RM}(\Omega)$.
Let
\[
\begin{aligned}
 \Sigma&:=H_0({\rm div};\Omega,\mathbb S)
 :=\{\tau\in H({\rm div};\Omega,\mathbb S):\tau n=0
       \text{ on }\partial\Omega\},\\
 V&:=L^2(\Omega,\mathbb V)/{\rm RM}(\Omega).
\end{aligned}
\]
For a Hu--Zhang-type mixed method, let
$\Sigma_h\subset\Sigma$ be the conforming stress space and let
$V_h\subset L^2(\Omega,\mathbb V)$ be the discontinuous displacement space,
understood modulo rigid motions.  The discrete solution
$(\sigma_h,u_h)\in\Sigma_h\times V_h$ is defined by
\begin{equation}\label{def:mixElasticity}
\begin{aligned}
(\mathcal{A}\sigma_h,\tau_h)+(\text{div}\tau_h,u_h)&=0,\quad\tau_h\in  \Sigma_h,\\
(\text{div}\sigma_h,v_h)&=-(f,v_h),\quad v_h\in V_h.
\end{aligned}
\end{equation}
We take the Hu--Zhang stress degree to be $p\geq4$, the discontinuous
displacement degree to be $p-1$, and set
\[
 f_h:=-{\rm div}\,\sigma_h.
\]

For convenience, let $J={\rm curlcurl}$ be the Airy stress operator.
We use the weighted inner products and norms
\[
 (\tau,\xi)_{\mathcal A}:=(\mathcal A\tau,\xi),
 \qquad \|\tau\|_{\mathcal A}^2:=(\tau,\tau)_{\mathcal A},
 \qquad
 \|\varepsilon(v)\|_{\mathcal C}^2
 :=(\mathcal C\varepsilon(v),\varepsilon(v)).
\]

\begin{lemma}\label{thm:Hodge_tensor}
Every $\tau\in L^2(\Omega,\mathbb S)$ admits a unique
$\mathcal A$-orthogonal decomposition
\begin{equation*}
 \tau=J\varphi
 +\mathcal C\varepsilon(w),
 \qquad
 \varphi\in H_0^2(\Omega),\quad
 [w]\in H^1(\Omega,\mathbb V)/{\rm RM}(\Omega).
\end{equation*}
Consequently,
\begin{equation*}
 \|\tau\|_{\mathcal A}^2
 =\|J\varphi\|_{\mathcal A}^2
 +\|\mathcal C\varepsilon(w)\|_{\mathcal A}^2.
\end{equation*}
\end{lemma}

\begin{proof}
Set $X:=H^1(\Omega,\mathbb V)/{\rm RM}(\Omega)$.
Since $\lambda,\mu>0$, Korn's inequality
\citep[Chap.~3, Thm.~3.4]{DuvautLions1976} makes
$(\mathcal C\varepsilon(\cdot),\varepsilon(\cdot))$ bounded and
coercive on $X$. Lax--Milgram therefore gives a unique $[w]\in X$ such that
\begin{equation}\label{eq:weighted-helmholtz-projection}
 (\mathcal C\varepsilon(w),\varepsilon(v))
 =(\tau,\varepsilon(v))\qquad\forall[v]\in X.
\end{equation}
Consequently, $\zeta:=\tau-\mathcal C\varepsilon(w)$ satisfies
$\operatorname{div}\zeta=0$ and $\zeta n=0$ in the weak sense.
Exactness of the Airy stress complex on the contractible domain
\citep[Sec.~2.1 and Lemma~2.3]{AznaranHuParker2026}
gives a unique $\varphi\in H_0^2(\Omega)$ with $\zeta=J\varphi$.
Moreover, $(J\varphi,\mathcal C\varepsilon(w))_{\mathcal A}
=(\zeta,\varepsilon(w))=0$. Uniqueness and the norm identity follow.
\end{proof}

For the stress error $e_\sigma:=\sigma-\sigma_h$,
Lemma~\ref{thm:Hodge_tensor} gives
\begin{equation}\label{eq:elasticity-error-decomposition}
 e_\sigma=J\varphi+\mathcal C\varepsilon(w),
 \qquad
 \|e_\sigma\|_{\mathcal A}^2
 =\|J\varphi\|_{\mathcal A}^2
  +\|\mathcal C\varepsilon(w)\|_{\mathcal A}^2.
\end{equation}
We estimate the Airy component by residual reconstruction and the
compatible-strain component by data oscillation.

Since $\sigma_h\in H({\rm div};\Omega,\mathbb S)$ and $\mu$ is constant,
the normal-tangential component of $\mathcal A\sigma_h$ is continuous
across interior edges:
\[
 \llbracket(\mathcal A\sigma_h)_{nt}\rrbracket
 =\frac1{2\mu}\llbracket\sigma_{h,nt}\rrbracket=0.
\]
Thus elementwise integration by parts gives, for $v\in H_0^2(\Omega)$,
\begin{align*}
 \langle R,v\rangle
 &:=(\mathcal A\sigma_h,Jv)\\
 &=({\rm rotrot}_h\mathcal A\sigma_h,v)
   +(\llbracket(\mathcal A\sigma_h)_{tt}\rrbracket,\partial_n v)_{\mathcal E_h}+(\llbracket t\cdot{\rm rot}_h\mathcal A\sigma_h\rrbracket,v)_{\mathcal E_h}.
\end{align*}
Using the residual representation above and the
localization rule in Section~\mbox{\ref{subsect:hminus2-equilibration}},
with $\phi_i v\in H_0^2(\Omega)$, we obtain
$R_i^o,R_i^n,R_i^b$, with $R_i^g=0$.
For an interior patch and $\ell\in\mathcal P_1(\Omega_i)$, let $\psi$
be the zero extension of $\phi_i\ell$. The Argyris boundary and vertex
conditions imply that $\psi\in H_0^2(\Omega)$ and $J\psi$ is
vertex-continuous and piecewise of degree at most four.
Thus $J\psi\in\Sigma_h$ for $p\geq4$, and \eqref{def:mixElasticity} gives
the required compatibility:
\begin{equation}\label{eq:elasticity-local-compatibility}
 \langle R,\phi_i\ell\rangle
 =(\mathcal A\sigma_h,J\psi)
 =-({\rm div}\,J\psi,u_h)=0.
\end{equation}
For reconstruction degree $k\geq p+5$, use the spaces and boundary
conditions of Section~\ref{sect:prelim} to compute
\begin{equation}\label{eq:elasticity-local-lift}
 q_i
 :=\operatorname*{argmin}_{\substack{
 q\in\Sigma_{h,k}(\Omega_i)\\
 {\rm divdiv}\,q=R_i}}
 \|q\|_{0,\Omega_i}.
\end{equation}
The unweighted minimum-$L^2$ reconstruction $q_i$ is computed
by \eqref{eq:sigmaiR}.
To relate the reconstruction to the Airy component, let
\[
 Q:=\begin{pmatrix}0&1\\-1&0\end{pmatrix},
 \qquad
 \mathscr P\tau:=Q\tau Q^\top.
\]
With the present convention for ${\rm curl}$,
\begin{equation}\label{eq:airy-rotation}
 Jv=\mathscr P\nabla^2v,
 \qquad
 J'\tau={\rm divdiv}(\mathscr P\tau).
\end{equation}

\begin{lemma}\label{lem:elasticity-airy-norm}
For every $v\in H_0^2(\Omega)$,
\begin{equation}\label{eq:elasticity-airy-norm}
 \|Jv\|_{\mathcal A}^2=\kappa|v|_{2,\Omega}^2,
 \qquad
 \kappa:=\frac{\lambda+2\mu}{4\mu(\lambda+\mu)}.
\end{equation}
\end{lemma}
\begin{proof}
Integration by parts and density give
$\|\Delta v\|^2=|v|_{2,\Omega}^2$ on $H_0^2(\Omega)$.
Since $\mathscr P$ preserves the Frobenius norm and trace,
\begin{align*}
 \|Jv\|_{\mathcal A}^2
 =\frac1{2\mu}\left(|v|_{2,\Omega}^2
       -\frac{\lambda}{2(\lambda+\mu)}\|\Delta v\|^2\right)=\kappa|v|_{2,\Omega}^2.
\end{align*}
The proof is complete. 
\end{proof}

Since $\sum_i\phi_i=1$, the assembled reconstruction satisfies
\begin{equation}\label{eq:elasticity-physical-lift}
 q_h:=\sum_{i\in\mathcal I}q_i,\qquad {\rm divdiv}\,q_h=R.
\end{equation}
Theorem~\ref{thm:p-robust-equilibration} gives
\begin{equation}\label{eq:elasticity-unweighted-stability}
 \|R\|_{H^{-2}(\Omega)}
 \leq\|q_h\|_{0,\Omega}
 \leq C(\gamma)\|R\|_{H^{-2}(\Omega)}.
\end{equation}
Here $C(\gamma)$ depends only on the mesh shape regularity,
and is independent of $h$, $p$, $\lambda$, and $\mu$.

Since $J'\mathcal A\sigma=J'\varepsilon(u)=0$, we have
$R=-J'\mathcal A e_\sigma$. Polarization of
\eqref{eq:elasticity-airy-norm} and the orthogonal decomposition give
\[
 \langle R,v\rangle
 =-(\mathcal A J\varphi,Jv)
 =-\kappa(\nabla^2\varphi,\nabla^2v),
 \qquad v\in H_0^2(\Omega).
\]
Hence
\begin{equation}\label{eq:elasticity-residual-norm}
 \|R\|_{H^{-2}(\Omega)}
 =\kappa|\varphi|_{2,\Omega}
 =\sqrt\kappa\,\|J\varphi\|_{\mathcal A}.
\end{equation}
Combining this identity with \eqref{eq:elasticity-unweighted-stability}
yields
\begin{equation}\label{eq:elasticity-airy-error-component}
 \|J\varphi\|_{\mathcal A}
 \leq\kappa^{-1/2}\|q_h\|_{0,\Omega}
 \leq C(\gamma)\|J\varphi\|_{\mathcal A}.
\end{equation}

Let $\Pi_{T,p-1}$ be the componentwise $L^2(T)$ projection onto
$\mathcal P_{p-1}(T,\mathbb V)$ and set $r_{f,T}:=(f-f_h)|_T$.
Equation~\eqref{def:mixElasticity} gives
$(r_{f,T},q)_T=0$ for all $q\in\mathcal P_{p-1}(T,\mathbb V)$.
Orthogonality in \eqref{eq:elasticity-error-decomposition},
$e_\sigma n=0$, and $-\operatorname{div}e_\sigma=f-f_h$ imply
\begin{equation}\label{eq:elasticity-strain-target}
 \|\mathcal C\varepsilon(w)\|_{\mathcal A}^2
 =(f-f_h,w)
 =\sum_{T\in\mathcal T_h}(r_{f,T},w-\Pi_{T,p-1}w)_T.
\end{equation}

For $\widehat T=\operatorname{conv}\{(0,0),(1,0),(0,1)\}$, write
$\widehat h=\operatorname{diam}(\widehat T)$ and let
$\widehat\Pi_{p-1}$ be the corresponding reference projection. Define
\begin{equation}\label{eq:elasticity-alpha-reference}
 \alpha_{\widehat T,p-1}
 :=\sup_{0\ne[\widehat v]\in H^1(\widehat T,\mathbb V)/{\rm RM}(\widehat T)}
 \frac{\|\widehat v-\widehat\Pi_{p-1}\widehat v\|_{\widehat T}}
 {\widehat h\,\|\varepsilon(\widehat v)\|_{\widehat T}},
 \qquad p\ge2.
\end{equation}
Since the projection preserves rigid motions, the quotient is well defined;
nesting gives $\alpha_{\widehat T,p}\leq\alpha_{\widehat T,p-1}$.
A computable upper bound is
\[
 \alpha_{\widehat T,p-1}\leq\overline\alpha_p
 :=\frac1{\sin(\pi/16)\sqrt{2p(p+2)}}.
\]

For $F_T(\widehat x)=B_T\widehat x+b_T$ and
$\widehat v=B_T^\top(v\circ F_T)$, projection commutation and the chain rule give
\[
 \widehat\Pi_{p-1}\widehat v=B_T^\top((\Pi_{T,p-1}v)\circ F_T),
 \qquad
 \varepsilon_{\widehat x}(\widehat v)
 =B_T^\top(\varepsilon_x(v)\circ F_T)B_T.
\]
Changing variables in \eqref{eq:elasticity-alpha-reference} yields
\begin{equation}\label{eq:elasticity-beta}
 \|v-\Pi_{T,p-1}v\|_T
 \leq\alpha_{\widehat T,p-1}\beta_T h_T\|\varepsilon(v)\|_T,
 \qquad
 \beta_T:=\frac{\widehat h}{h_T}\|B_T^{-1}\|_2\|B_T\|_2^2.
\end{equation}
If $T$ is similar to $\widehat T$, $F_T$ may be chosen so that $\beta_T=1$.
Since $\mathcal C\tau:\tau\geq2\mu|\tau|^2$, set
\begin{align*}
 \gamma_{T,p}&:=\frac{\alpha_{\widehat T,p-1}\beta_T}{\sqrt{2\mu}},
 &\operatorname{osc}_T(f)&:=\gamma_{T,p}h_T\|f-f_h\|_{0,T},\\
 \operatorname{osc}_h(f)^2&:=\sum_{T\in\mathcal T_h}\operatorname{osc}_T(f)^2.
\end{align*}
Thus $\|v-\Pi_{T,p-1}v\|_T
\leq\gamma_{T,p}h_T\|\varepsilon(v)\|_{\mathcal C,T}$.
Cauchy--Schwarz in \eqref{eq:elasticity-strain-target}, together with
$\|\varepsilon(w)\|_{\mathcal C}
=\|\mathcal C\varepsilon(w)\|_{\mathcal A}$, proves
\begin{equation}\label{eq:elasticity-strain-component-bound}
 \|\mathcal C\varepsilon(w)\|_{\mathcal A}\leq\operatorname{osc}_h(f).
\end{equation}
The weights are independent of $\lambda$; replacing
$\alpha_{\widehat T,p-1}$ by $\overline\alpha_p$ preserves this bound.

Using the reconstruction \eqref{eq:elasticity-local-lift}, define
\begin{align*}
 \eta_T&:=\left(\kappa^{-1}\|q_h\|_{0,T}^2
                    +\operatorname{osc}_T^2(f)\right)^{1/2},\\
 \eta_h&:=\left(\sum_{T\in\mathcal T_h}\eta_T^2\right)^{1/2}
 =\left(\kappa^{-1}\|q_h\|_{0,\Omega}^2
                    +\operatorname{osc}_h^2(f)\right)^{1/2}.
\end{align*}
For patchwise efficiency, set
\begin{equation*}
 \operatorname{osc}_i^2(f):=\sum_{T\in\mathcal T_i}
                   \operatorname{osc}_T^2(f),\qquad
 \eta_i^{\rm el}:=\left(\kappa^{-1}\|q_i\|_{0,\Omega_i}^2
                         +\operatorname{osc}_i^2(f)\right)^{1/2}.
\end{equation*}
\begin{theorem}\label{thm:elasticity-estimator}
Let $(\sigma_h,u_h)$ be the Hu--Zhang FEM solution of
\eqref{def:mixElasticity} with stress degree $p\ge4$ and positive
constant Lam\'e parameters $\lambda$ and $\mu$.
With reconstruction degree $k\ge p+5$,
the equilibrated residual error estimator satisfies
\[
 \|\sigma-\sigma_h\|_{\mathcal A}
 \leq\eta_h
 \leq C(\gamma)\|\sigma-\sigma_h\|_{\mathcal A}
      +\operatorname{osc}_h(f).
\]
Consequently, the estimator is simultaneously $p$- and
$\lambda$-robust up to data oscillation. 
The patch reconstructions and patch indicators satisfy
\begin{align*}
 \kappa^{-1/2}\|q_i\|_{0,\Omega_i}
 &\leq C(\gamma_i)
              \|\sigma-\sigma_h\|_{\mathcal A,\Omega_i},\\
 \eta_i^{\rm el}
 &\leq C(\gamma_i)
       \|\sigma-\sigma_h\|_{\mathcal A,\Omega_i}
       +\operatorname{osc}_i(f).
\end{align*}
Here $C(\gamma)$ and $C(\gamma_i)$ depend only on the corresponding
mesh shape regularity and are independent of $h$, $p$, $\lambda$, and $\mu$.
\end{theorem}

\begin{proof}
The orthogonal decomposition \eqref{eq:elasticity-error-decomposition},
the Airy bound \eqref{eq:elasticity-airy-error-component}, and
\eqref{eq:elasticity-strain-component-bound} give reliability.
For efficiency, the same bounds yield
\[
 \eta_h\leq\kappa^{-1/2}\|q_h\|_{0,\Omega}+\operatorname{osc}_h(f)
 \leq C(\gamma)\|e_\sigma\|_{\mathcal A}+\operatorname{osc}_h(f).
\]
For the local bounds, take $v\in X_i$ and
$z=v-\ell_i(v)$ as in \eqref{eq:p-robust-patch-poincare}.
Let $\psi_i\in H_0^2(\Omega)$ be the zero extension of $\phi_i z$.
The residual identity and affine compatibility give
\begin{align*}
 |R_i(v)|&=|(\mathcal A e_\sigma,J\psi_i)_{\Omega_i}|\\
 &\leq\|e_\sigma\|_{\mathcal A,\Omega_i}
             \|J\psi_i\|_{\mathcal A,\Omega_i}\\
 &=\sqrt\kappa\,\|e_\sigma\|_{\mathcal A,\Omega_i}
             |\psi_i|_{2,\Omega_i}\\
 &\leq C(\gamma_i)\sqrt\kappa\,
             \|e_\sigma\|_{\mathcal A,\Omega_i}|v|_{2,\Omega_i}.
\end{align*}
The equality uses \eqref{eq:elasticity-airy-norm} for the zero-extended
function $\psi_i$, and the last inequality uses the product estimate in
Lemma~\ref{lem:p-robust-localization}.
Taking the dual norm and applying
Theorem~\ref{thm:p-robust-local-lifting} proves the bound for $q_i$.
The patch-indicator bound follows from
$\eta_i^{\rm el}\leq\kappa^{-1/2}\|q_i\|_{0,\Omega_i}
+\operatorname{osc}_i(f)$.
\end{proof}

\begin{remark}
\label{rem:elasticity-constant-parameters}
The explicit factor satisfies $2\mu<\kappa^{-1}<4\mu$ and converges
to $4\mu$ as $\lambda\to\infty$ with $\mu$ fixed. Thus the estimator
remains reliable and efficient in the nearly incompressible limit.
\end{remark}

\section{Numerical Experiments}\label{sect:numerical}
We illustrate the estimators for the three model problems.
Adaptive mesh
refinement is based on the standard D\"orfler marking strategy with marking
parameter $\theta=0.4$ and newest vertex bisection. In the following tests the
effectivity ratio is computed as $\eta_h/{\rm err}$, where ${\rm err}$ denotes
the corresponding reference error.

\subsection{Argyris FEM for the biharmonic equation}
Let $\Omega={\rm conv}\{(\cos(2\pi j/5),\sin(2\pi j/5)):\ j=0,\ldots,4\}.
$
We consider the clamped biharmonic problem
$$
    \Delta^2 u=f\quad\text{in }\Omega,\qquad
    u=\partial_n u=0\quad\text{on }\partial\Omega .
$$
We set $f=1$ and test the adaptive Argyris method of polynomial degree
$p=6,7,8$.  The equilibrated estimator is obtained by solving the hybridized
vertex-patch $H({\rm divdiv})$ reconstruction problem with the Argyris partition
of unity and local polynomial degree $p+3$.  The initial mesh size is $h=0.5$,
and the refinement is stopped when the number of degrees of freedom reaches
$10^4$.  Since the exact solution is not available, the $u$ in the error
$|u-u_h|_{H^2(\Omega)}$ is approximated by a reference solution in the Argyris
space of degree $p+1$ on the same mesh;
the comparison uses the continuous $H^2(\Omega)$ seminorm.

\begin{figure}[!htbp]
    \centering
    \includegraphics[width=0.7\textwidth]{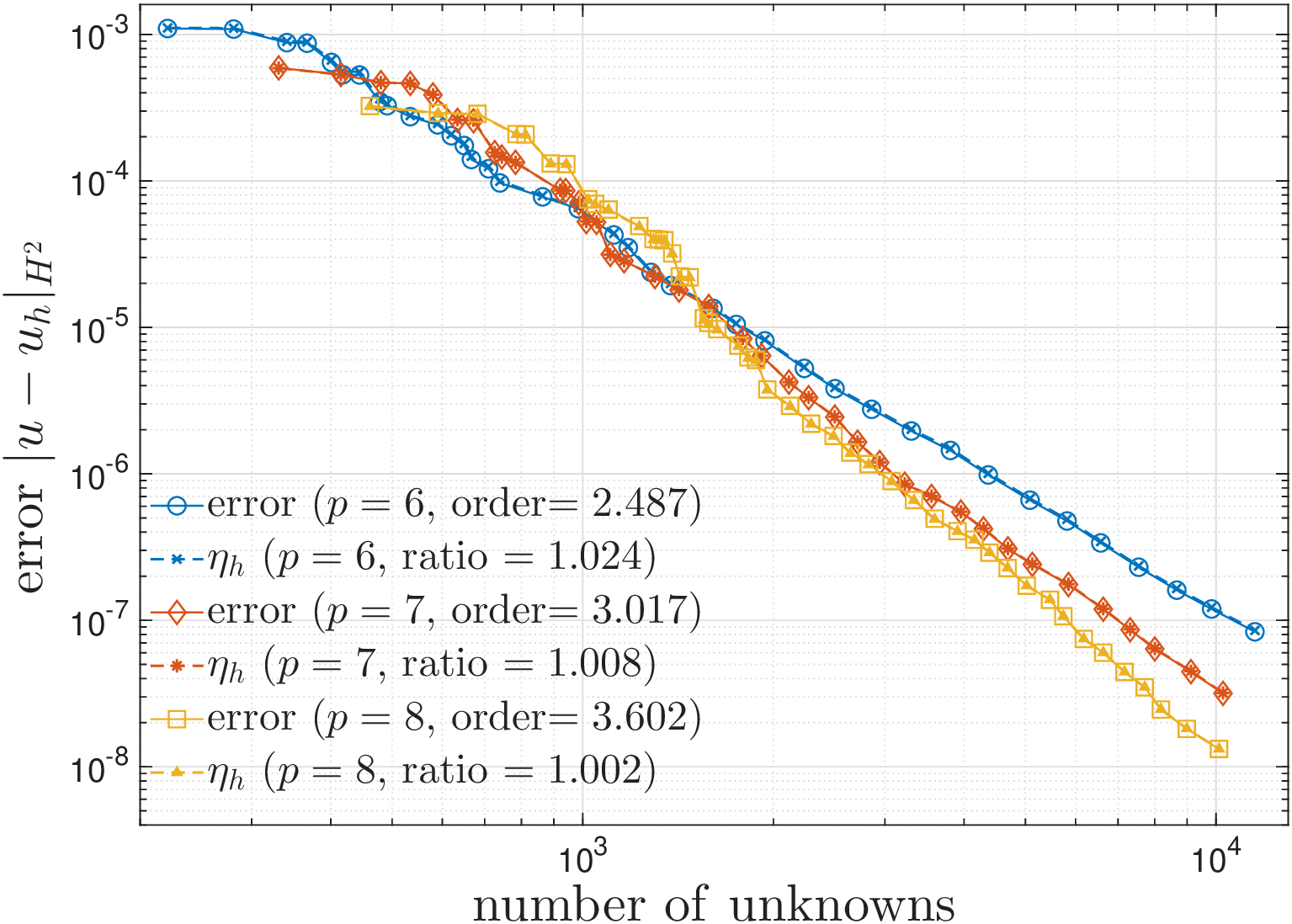}
    \caption{Adaptive Argyris FEM for the biharmonic equation.}
    % Alt text: Adaptive Argyris finite element results for the clamped biharmonic equation on a regular pentagon, comparing polynomial degrees $p=6,7,8$.
    \label{fig:biharmonic-h2-yebi}
\end{figure}

\subsection{Conforming FEM for Poisson's equation}
On the same regular pentagon $\Omega$, we consider the Poisson equation
$$
    -\Delta u=f\quad\text{in }\Omega,\qquad
    u=0\quad\text{on }\partial\Omega .
$$
We set $f=1$ and test conforming Lagrange finite elements of polynomial degree
$p=2,3,4,5$.  The $L^2$ error estimator is given by the local
$H({\rm divdiv})$ equilibrated reconstruction in
Section~\ref{sect:l2-error-estimate}; the
reconstructed tensor has degree $p+1$.
The initial mesh size is $h=0.25$, and the refinement is stopped when the number
of degrees of freedom reaches $10^5$.  The exact solution is again not
available, and the $u$ in the $L^2$ error is approximated by a reference Lagrange solution
of degree $p+1$ on the same mesh.

% Previously fixed with [H]; allow floating to avoid stretched page spacing.
\begin{figure}[!htbp]
    \centering
    \includegraphics[width=0.7\textwidth]{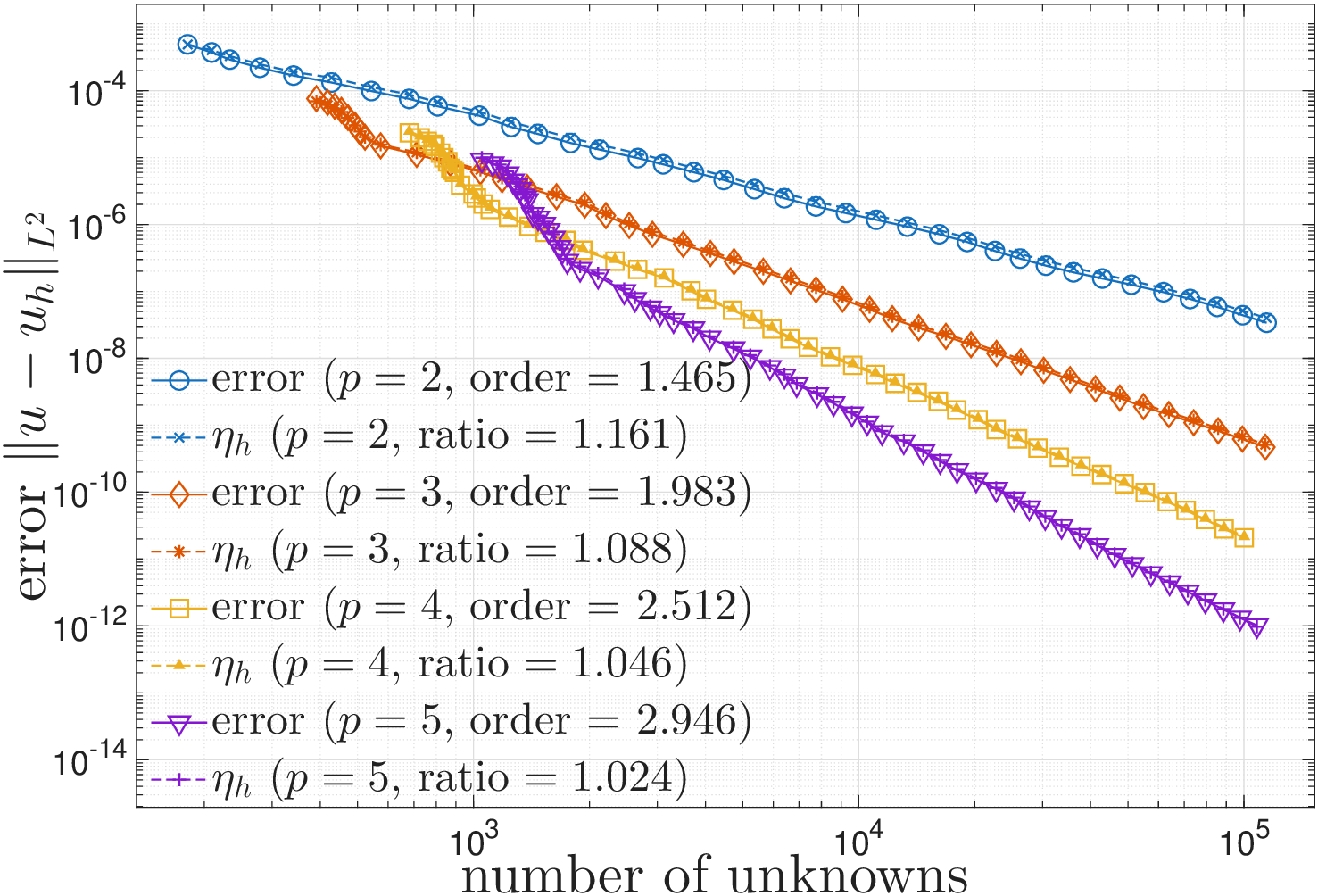}
    \caption{Adaptive Lagrange FEM for the Poisson equation.}
    % Alt text: Adaptive Lagrange finite element results for Poisson's equation on a regular pentagon, comparing polynomial degrees $p=2,3,4,5$ using the $L^2$ error estimator.
    \label{fig:Poisson L2}
\end{figure}

\subsection{Hu--Zhang FEM for linear elasticity}
Throughout this subsection, including
Figure~\ref{fig:elasticity},
$\eta_h$ denotes the estimator of Theorem~\ref{thm:elasticity-estimator}.

% Place the subfigures side by side.
\begin{figure}[!htbp]
    \centering
    \subfloat{%
        \includegraphics[width=0.45\textwidth]{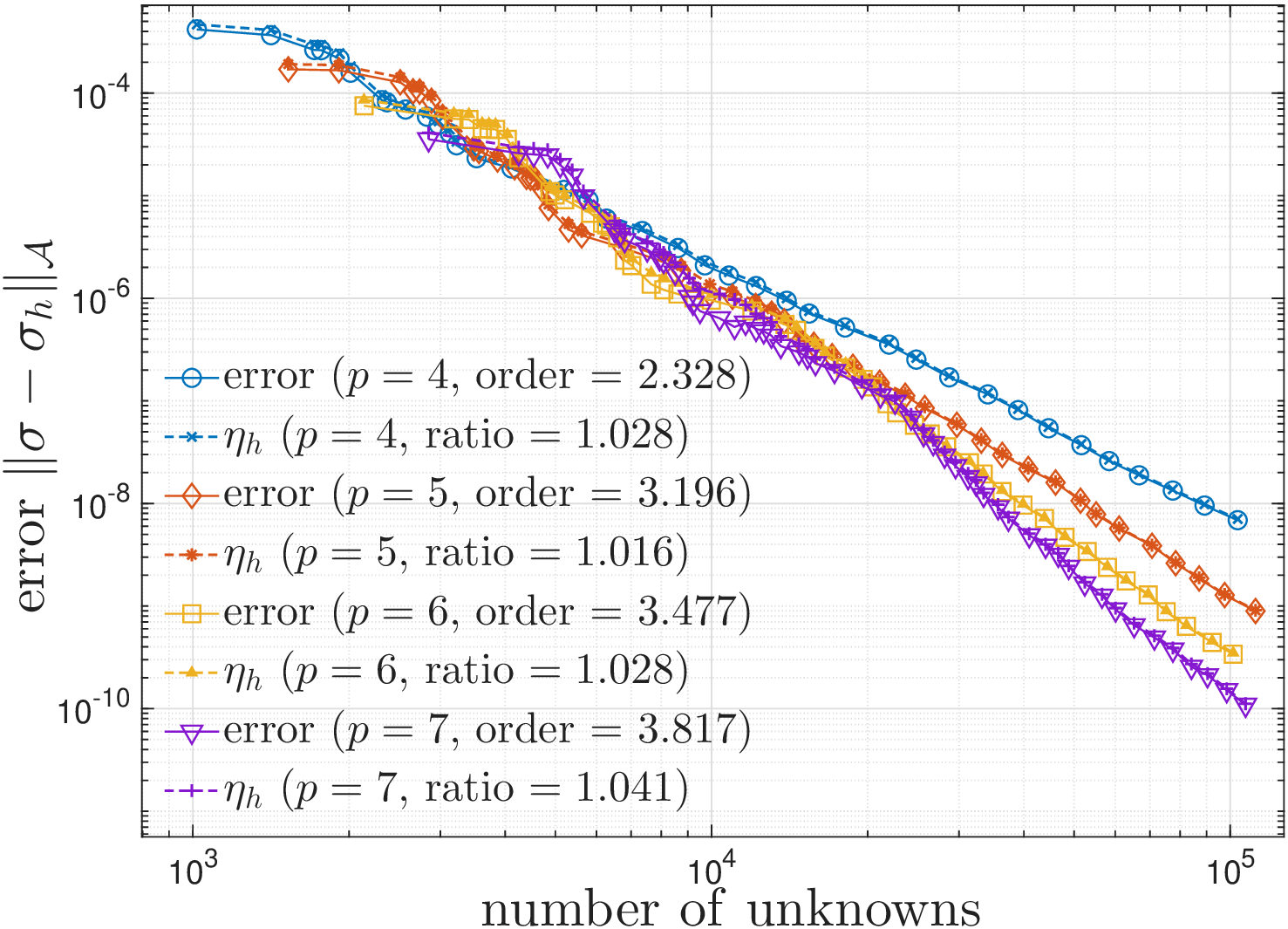}}
    \hfill
    \subfloat{%
        \includegraphics[width=0.45\textwidth]{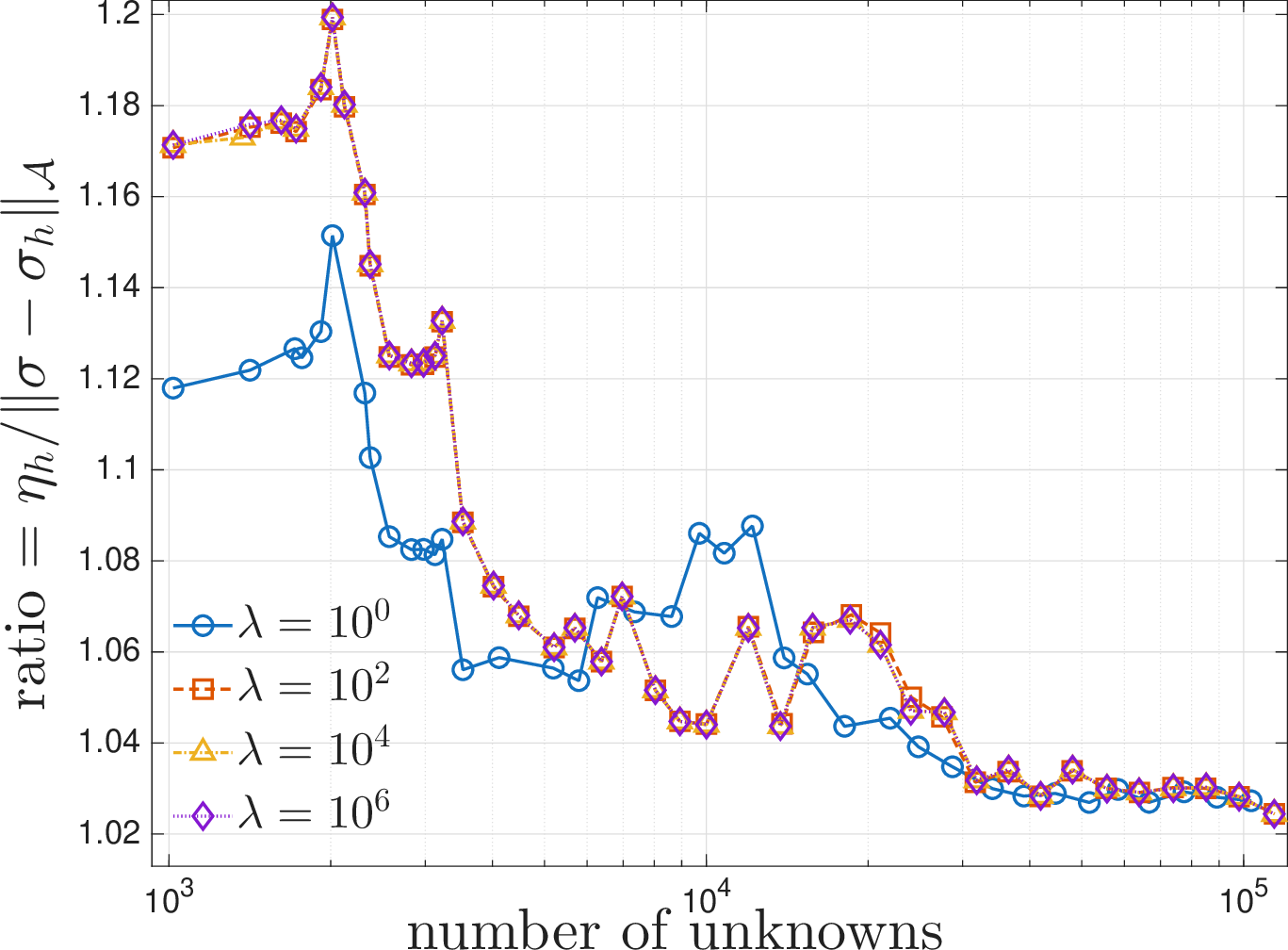}}
    \caption{Hu--Zhang mixed FEM for linear elasticity: adaptive convergence
    for $p=4,5,6,7$ (left) and $\lambda$-robustness for $p=4$ and $\mu=1$ (right).}
    % Alt text: Two panels for the Hu--Zhang mixed finite element method: adaptive convergence for stress degrees $p=4,5,6,7$ on the left, and effectivity ratios for $p=4$, $\mu=1$, and $\lambda=1,10^2,10^4,10^6$ on the right.
    \label{fig:elasticity}
\end{figure}
On the same regular pentagon $\Omega$, we consider the linear elasticity
problem
$$
    \mathcal A\sigma
    =\varepsilon(u),
    \qquad
    -\operatorname{div}\sigma
    =f
    \quad\text{in }\Omega,
    \qquad
    \sigma n
   =0
    \quad\text{on }\partial\Omega .
$$
We take $\lambda=\mu=1$ and prescribe
$f(x,y)
    = (x,y)^\top.$
Since the centroid of the regular pentagon is the origin,
$\int_\Omega f\,dx=0.$
Moreover, $f$ satisfies the compatibility condition for the
pure-traction problem.

We test the Hu--Zhang mixed finite element method with stress polynomial
degrees 
$p=4,5,6,7$
and discontinuous displacement degree $p-1$.
The equilibrated estimator is obtained from the vertex-patch
$H({\rm divdiv})$ reconstruction.
We use reconstruction degree $k=p+5$ in all elasticity tests.
Since $f\in\mathcal P_{p-1}
(\mathcal T_h;\mathbb R^2)$, the discrete equilibrium equation gives
$f_h=f$ and hence $\operatorname{osc}_h(f)=0$.
The initial mesh size is $h=0.5$, and refinement is stopped when the
number of degrees of freedom reaches $10^5$. Since the exact solution is
not available, the stress error
$\|\sigma-\sigma_h\|_{\mathcal A}$ is
approximated by
$
    \left\|
        \sigma_h^{p+1}
        -
        \sigma_h^p
    \right\|_{\mathcal A},
$
where the reference Hu--Zhang solution of degree $p+1$ is computed on the
same mesh.

To examine parameter robustness, we fix the stress degree $p=4$ and the
shear modulus $\mu=1$, and vary $\lambda\in\{1,10^2,10^4,10^6\}$.
Figure~\ref{fig:elasticity} (right) shows the effectivity ratios
$\eta_h/{\rm err}$, using the reference stress error defined above, illustrating the estimator's
$\lambda$-robust behavior.

\section*{Funding}
This work was supported by the National Natural Science Foundation of China under grant 12471346.

\section*{Acknowledgments}
The authors thank Xuehai Huang for discussions on discrete
$\operatorname{div}\operatorname{div}$ complexes with boundary
conditions, and Qian Zhang and T. Chaumont-Frelet for helpful discussions.
The authors also thank Zhen Liu for pointing out an unnecessary vertex
constraint in an earlier version.

The authors used GPT-5.6 sol and -6.0 Astra (OpenAI) to turn a polynomial-degree-dependent efficiency proof in Appendix~\ref{app:p-robust-lifting} into a degree-independent one. In addition, Codex assisted with parts of the experimental code. The authors have reviewed the code, verified and rewritten all
AI-assisted theoretical parts, and take full responsibility for the manuscript.

\bibliographystyle{IMANUM-BIB}

% Restore the manuscript running heads after the references.
\makeatletter
\markboth{\@shortauthorlist}{\@shorttitle}
\makeatother
\appendix
% The class assumes one appendix; use the section letter for both appendices.
\renewcommand{\theequation}{\thesection.\arabic{equation}}
\renewcommand{\thetheorem}{\thesection.\arabic{theorem}}
\renewcommand{\thelemma}{\thesection.\arabic{lemma}}
\renewcommand{\thecorollary}{\thesection.\arabic{corollary}}

\begingroup
\section{Polynomial-degree-independent local lifting}
\label{app:p-robust-lifting}
\newcommand{\dd}{\operatorname{div}\operatorname{div}}
\newcommand{\tr}{\operatorname{tr}}
\newcommand{\symcurl}{\operatorname{symcurl}}
\newcommand{\jump}[1]{\llbracket #1\rrbracket}
\allowdisplaybreaks[1]

Throughout the appendices, we write $C_1\lesssim C_2$ if
$C_1\leq C_3C_2$, where $C_3>0$ is independent of the mesh size and
polynomial degrees. The constant may depend on the mesh shape regularity.
We write $C_1\eqsim C_2$ if $C_1\lesssim C_2$ and $C_2\lesssim C_1$.

We prove local lifting estimates with polynomial-degree-independent
constants. Appendix~\ref{app:p-robust-localization} establishes the corresponding
localization estimates. Under the hypotheses of the
corresponding results, minimum-norm reconstructions satisfy
\begin{equation*}
 \Bigl\|\sum_i\sigma_i\Bigr\|_{0,\Omega}^2
 \overset{\text{\tiny finite overlap}}{\le}
 3\sum_i\|\sigma_i\|_{0,\Omega_i}^2
 \underset{\text{\ref{thm:p-robust-local-lifting}}}{\lesssim}
 \begin{cases}
 \displaystyle
 \sum_i\|R_i\|_{X_i'}^2
 \underset{\text{\ref{lem:p-robust-localization}}}{\lesssim}
 \|R\|_{H^{-2}(\Omega)}^2,\\[1ex]
 \displaystyle
 \sum_i\|R_i\|_{Y_i'}^2
 \underset{\text{\ref{thm:poisson-p-localization}}}{\lesssim}
 \|u-u_h\|_{0,\Omega}^2.
 \end{cases}
\end{equation*}
The $X_i'$ branch applies to biharmonic and elasticity, and the $Y_i'$ branch
to Poisson. For elasticity, replace $\sigma_i$ by $q_i$.
Theorem~\ref{thm:p-robust-local-lifting} uses
Lemmas~\ref{lem:p-robust-interior-lifting} and~\ref{lem:p-robust-fan},
while Theorem~\ref{thm:poisson-p-localization} uses
Lemmas~\ref{lem:poisson-p-error-trace}--\ref{lem:poisson-p-smoothing}.

\subsection{Patch spaces and a common lifting estimate}
For the one-ring patches of Section~\ref{sect:prelim}, shape regularity gives
\begin{gather*}
 \#\mathcal T_i\leq C(\gamma_i),\qquad
 h_T\eqsim h_e\eqsim h_i.
\\
\shortintertext{Let $\Gamma_{\rm out}$ collect the edges opposite $a_i$, and let
$\Gamma_{\rm rad}$ collect the two physical radial edges of a boundary
patch ($\Gamma_{\rm rad}=\varnothing$ on an interior patch). Set}
 X_i:=\begin{cases}
 H^2(\Omega_i)/\mathcal P_1(\Omega_i),&a_i\in\mathring\Omega,\\
 \{v\in H^2(\Omega_i):v=\partial_nv=0\text{ on }\Gamma_{\rm rad}\},
     &a_i\in\partial\Omega,
 \end{cases}
\end{gather*}
\begin{equation*}
 Y_i:=\{v\in H^2(\Omega_i):v|_{\Gamma_{\rm rad}}=0\}/\mathcal L_i,
 \quad\text{in which }\mathcal L_i:=\{\ell\in\mathcal P_1(\Omega_i):
 \ell|_{\Gamma_{\rm rad}}=0\}.
\end{equation*}
We use $X_i$ for residual lifting and localization in the
biharmonic and elasticity problems, and $Y_i$ for the Poisson $L^2$
error estimate. Their boundary conditions on $\Gamma_{\rm rad}$ reflect the global
test spaces $H_0^2(\Omega)$ and $H^2(\Omega)\cap H_0^1(\Omega)$,
respectively.
We equip $X_i$ and $Y_i$ with the norms
$\|v\|_{X_i}:=|v|_{2,\Omega_i}$ and
$\|v\|_{Y_i}:=|v|_{2,\Omega_i}$.
Traces on $\Gamma_{\rm out}$ are unrestricted, including on $\Gamma_{\rm out}\cap\partial\Omega$.
Fix $W_i\in\{X_i,Y_i\}$. For any representative $v\in W_i$, let
$\ell_i(v)$ be its $L^2(\Omega_i)$-orthogonal projection onto
$\mathcal M_i(W_i)$ defined in~\eqref{eq:p-robust-patch-poincare}.
Affine subtraction and successive
trace--Poincar\'e estimates across $\Omega_i$ give
\begin{equation}\label{eq:p-robust-patch-poincare}
 \begin{gathered}
 z=v-\ell_i(v),\qquad
 h_i^{-2}\|z\|_{0,\Omega_i}+h_i^{-1}\|\nabla z\|_{0,\Omega_i}
 \leq C_i|v|_{2,\Omega_i},\\
 \mathcal M_i(W_i):=\begin{cases}
 \mathcal P_1(\Omega_i),&W_i=X_i,\ a_i\in\mathring\Omega,\\
 \{0\},&W_i=X_i,\ a_i\in\partial\Omega,\\
 \mathcal L_i,&W_i=Y_i,
 \end{cases}\quad
 C_i=\begin{cases}C(\gamma_i),&W_i=X_i,\\C(\gamma_i,\Omega),&W_i=Y_i.
 \end{cases}
 \end{gathered}
\end{equation}
For boundary $X_i$, the conditions $v=\partial_n v=0$ on
$\Gamma_{\rm rad}$ exclude nonzero affine functions. In $Y_i$, the
quotient removes the affine functions in $\mathcal L_i$.
Estimate~\eqref{eq:p-robust-patch-poincare} therefore ensures completeness
in the Hessian norms. All constants below are independent of $h_i$ and
the polynomial degrees, while $C(\gamma_i,\Omega)$ may depend on the
fixed corner angles of $\Omega$.

\begin{theorem}
\label{thm:p-robust-local-lifting}\label{thm:poisson-p-lifting}
Let $k\geq3$. For the residual representation of
Section~\ref{subsect:hminus2-equilibration}, assume zero data on $\Gamma_{\rm out}$ and
$(r_i^o,r_i^n,r_i^b,r_i^g)\in U_{h,k}(\Omega_i)$.
If $r_i(\mathcal P_1)=0$ on interior patches, there exists a lifting such that
\begin{gather*}
 \operatorname{divdiv}\tau_{i,k}=r_i,\\
 \tau_{i,k}\in\Sigma_{h,k}(\Omega_i),\qquad
 \|\tau_{i,k}\|_{0,\Omega_i}\leq C(\gamma_i)\|r_i\|_{X_i'}.
\\
\shortintertext{If additionally $r_i^n=r_i^g=0$ and $r_i(\mathcal L_i)=0$, then, on a
fixed convex polygon, there exists such a lifting with}
 \tau_{i,k}\in\Sigma_{h,k}^{L^2}(\Omega_i),\qquad
 \|\tau_{i,k}\|_{0,\Omega_i}\leq C(\gamma_i,\Omega)\|r_i\|_{Y_i'}.
\end{gather*}
These bounds also hold for the minimum-norm reconstructions in the
stated spaces.
\end{theorem}

This theorem will be proved below using
Lemmas~\ref{lem:p-robust-interior-lifting} and~\ref{lem:p-robust-fan}.

\subsection{Element liftings and polynomial extension}
With $Q$ and $\mathscr P$ from \eqref{eq:airy-rotation},
\begin{gather*}
 \symcurl v:=-\mathscr P\varepsilon(Qv),\qquad \dd\symcurl v=0.
\end{gather*}

The trace identities
\begin{equation}\label{eq:p-robust-potential-boundary}
 \tr_1\symcurl v=-\partial_t(v\cdot n),\qquad
 \tr_2\symcurl v=-\partial_t^2(v\cdot t)
\end{equation}
hold distributionally on every open edge.
\begin{lemma}
\label{lem:p-robust-interior-lifting}\label{lem:p-robust-potential}
For every $\tau\in L^2(T,\mathbb S)$ with
$\dd\tau=g\in\mathcal P_{k-2}(T)$, there exist $\tau^g$ and $v$ such that
\begin{gather*}
 \begin{gathered}
 \tau=\tau^g+\symcurl v,\qquad
 \tau^g\in\mathcal P_k(T,\mathbb S),\quad v\in H^1(T,\mathbb R^2),\\
 \dd\tau^g=g,\qquad
 \|\tau^g\|_{0,T}+\|\nabla v\|_{0,T}\leq C(\gamma_T)\|\tau\|_{0,T}.
 \end{gathered}
\end{gather*}
\end{lemma}
\begin{proof}
On the reference triangle, the regularized divergence inverse satisfies
\citep[Corollary~3.4 and Proposition~4.2]{CostabelMcIntosh2010}
\begin{gather*}
 \operatorname{div}\mathcal Bg=g,\qquad
 \mathcal B\mathcal P_j\subset\mathcal P_{j+1}^2,\qquad
 \|\mathcal Bg\|_{H^{s+1}}\leq C\|g\|_{H^s},\quad s=-2,-1.
\end{gather*}
Using this operator, we construct $M$ by
$M_{j\ell}:=\bigl(\mathcal B((\mathcal Bg)_j)\bigr)_\ell$ for $j,\ell=1,2$
and set $\tau^g:=\tfrac12(M+M^\top)$. Then $\dd\tau^g=\dd M=g$, and
\begin{equation*}
 \|\tau^g\|_0\leq\|M\|_0\leq C\|g\|_{H^{-2}}\leq C\|\tau\|_0.
\end{equation*}
For $F(\widehat x)=B\widehat x+b$ and $J_B=|\det B|$,
we use the matrix Piola transformation
\citep[Section~4]{AW2002}, normalized by $J_B^{-1}$:
\begin{equation}\label{eq:merged-double-piola}
 \begin{gathered}
 \widehat\tau=J_BB^{-1}(\tau\circ F)B^{-\top},\qquad
 \widehat g=J_B(g\circ F),\qquad
 \tau^g\circ F=J_B^{-1}B\widehat\tau^{\,\widehat g}B^\top,\\
 \|\tau^g\|_{0,T}\leq C\|B\|_2^2\|B^{-1}\|_2^2\|\tau\|_{0,T}
 \leq C(\gamma_T)\|\tau\|_{0,T}.
 \end{gathered}
\end{equation}
Since $\dd\rho=0$ for $\rho=\tau-\tau^g$, exactness of the
$\dd$ complex on $T$ \citep[Sec.~4.5]{ArnoldHu2021} gives
$\tilde v\in H^1(T,\mathbb R^2)$ with $\symcurl\tilde v=\rho$.
After subtracting a suitable element of $\ker(\symcurl)$ from
$\tilde v$, we denote the resulting potential by $v$.
Korn's inequality applied to $Qv$ yields
\begin{equation*}
 \|\nabla v\|_{0,T}
 \leq C(\gamma_T)\|\varepsilon(Qv)\|_{0,T}
 =C(\gamma_T)\|\rho\|_{0,T}.
\end{equation*}
\end{proof}

For the residual in Theorem~\ref{thm:p-robust-local-lifting},
the function $v\in H^1(\mathcal T_i,\mathbb R^2)$ obtained by applying
Lemma~\ref{lem:p-robust-interior-lifting} elementwise satisfies
$\jump{v}|_e\in\mathcal P_{k+1}(e)^2$ on every interior edge $e$, as
shown in the proof of Theorem~\ref{thm:p-robust-local-lifting} below.

To lift these polynomial edge jumps componentwise, we use
the single-edge lifting operators in
\citet[Lemmas~4.1--4.3]{ParkerSuli2023}, which, for
$W\in H^1(T)$ with polynomial traces of degree at most $m$
on one or two edges $\Gamma$, give
\begin{equation}\label{eq:p-robust-partial-extension}
 z\in\mathcal P_m(T),\qquad z|_\Gamma=W|_\Gamma,\qquad
 \|\nabla z\|_{0,T}\leq C(\gamma_T)\|\nabla W\|_{0,T}.
\end{equation}
Estimate~\eqref{eq:p-robust-partial-extension} follows from
the $H^1$-stable polynomial trace liftings in
\citet[Lemmas~4.1--4.3]{ParkerSuli2023}, together with the Poincar\'e
inequality. The two-edge case is obtained by applying these liftings
successively.

For the edge corrections below, we also need a lifting
from one edge that vanishes on the other two edges.
Let $e\subset\partial T$ and
$\delta_e(s):=\operatorname{dist}(s,\partial e)$.
For $f\in\mathcal P_m(e)$ with $f|_{\partial e}=0$,
\citet[Lemma~4.3]{ParkerSuli2023} gives a lifting $z_f$ with
\begin{equation}\label{eq:p-robust-zero-edge-extension}
 \begin{gathered}
 z_f\in\mathcal P_m(T),\qquad z_f|_e=f,\qquad
 z_f|_{\partial T\setminus e}=0,\\
 \|\nabla z_f\|_{0,T}^2
 \leq C(\gamma_T)\Bigl(h_e^{-1}\|f\|_{0,e}^2
 +|f|_{H^{1/2}(e)}^2
 +\int_e\frac{|f(s)|^2}{\delta_e(s)}\,\mathrm ds\Bigr).
 \end{gathered}
\end{equation}
The constant is independent of $h_T$ and $m$.
The construction applies componentwise to vector data.

\begin{lemma}
\label{lem:p-robust-fan}\label{lem:poisson-p-fan}
Let $v\in H^1(\mathcal T_i,\mathbb R^2)$ have polynomial
vector jumps of degree at most $m\geq1$ on interior edges and
$\Gamma_{\rm out}$.
There exists $v_m\in\mathcal P_m(\mathcal T_i,\mathbb R^2)$ such that
\begin{gather*}
 \jump{v_m}=\jump v
 \quad\text{on interior edges and }\Gamma_{\rm out},\\
 \|\nabla_hv_m\|_0\leq C(\gamma_i)\|\nabla_hv\|_0.
\end{gather*}
On a boundary patch, if $v\cdot n|_e\in\mathcal P_m(e)$
for every edge $e\subset\Gamma_{\rm rad}$, then $v_m$ can be chosen
to also satisfy $v_m\cdot n=v\cdot n$ on $\Gamma_{\rm rad}$.
The stability estimate for a fixed convex polygon $\Omega$ then reads
\[
 \|\nabla_hv_m\|_0\leq C(\gamma_i,\Omega)\|\nabla_hv\|_0.
\]
\end{lemma}
\begin{proof}
We first remove the prescribed outer traces.
Applying \eqref{eq:p-robust-partial-extension} elementwise gives
$b\in\mathcal P_m(\mathcal T_i,\mathbb R^2)$ with
\begin{gather*}
 b|_T=b_T\in\mathcal P_m(T)^2,\qquad w:=v-b,\quad
 w|_{\Gamma_{\rm out}}=0,\quad
 \|\nabla_hw\|_0\leq C\|\nabla_hv\|_0.
\\
\shortintertext{For $a_i\in\partial\Omega$, order the triangles as
$T_j=\operatorname{conv}\{a_i,b_{j-1},b_j\}$.
Let $F_j:T_{j+1}\to T_j$ fix the shared edge $e_j$ and map the
outer edge of $T_{j+1}$ to that of $T_j$.
Write $w_j:=w|_{T_j}$.
Starting with $z_1=0$, define}
 W_{j+1}:=w_{j+1}+(z_j-w_j)\circ F_j.
 \\
\shortintertext{Its traces on the outer and shared edges are, respectively,}
 W_{j+1}|_{\overline{b_jb_{j+1}}}=0,\qquad
 W_{j+1}|_{e_j}=z_j+(w_{j+1}-w_j)\in\mathcal P_m(e_j)^2.
\end{gather*}
Applying \eqref{eq:p-robust-partial-extension}
componentwise to these two traces gives
$z_{j+1}\in\mathcal P_m(T_{j+1},\mathbb R^2)$ with the same traces and
\begin{equation}\label{eq:p-robust-fan-recursion}
 \begin{aligned}
 \|\nabla z_{j+1}\|_{0,T_{j+1}}
 &\leq C\|\nabla W_{j+1}\|_{0,T_{j+1}}\\
 &\leq C\bigl(\|\nabla w_{j+1}\|_{0,T_{j+1}}
            +\|\nabla w_j\|_{0,T_j}+\|\nabla z_j\|_{0,T_j}\bigr).
 \end{aligned}
\end{equation}
Define the broken polynomial
$z\in\mathcal P_m(\mathcal T_i,\mathbb R^2)$ by $z|_{T_j}:=z_j$.
Then $z|_{\Gamma_{\rm out}}=0$, and the shared-edge trace relation above
gives $\jump z=\jump w$ on every interior edge. Moreover,
using \eqref{eq:p-robust-partial-extension} and iterating
\eqref{eq:p-robust-fan-recursion}, with
$\#\mathcal T_i\leq C(\gamma_i)$, yield
$\|\nabla_hz\|_0\leq C(\gamma_i)\|\nabla_hw\|_0$.

For $a_i\in\mathring\Omega$, cut along a radial edge $e_0$ and apply
the same induction. Define $E:=z-w$ and $d:=\jump E|_{e_0}$. Then $E$ is
continuous on $\Omega_i\setminus e_0$ and vanishes on
$\Gamma_{\rm out}$. Using polar coordinates $(r,\theta)$ centered at either
endpoint of $e_0$, angular integration
followed by trace--Poincar\'e yields,
with $r_0$ chosen within the corresponding endpoint sector,
\begin{equation*}
 \begin{gathered}
 |d(r)|^2\leq C\int |\partial_\theta E(r,\theta)|^2\,\mathrm d\theta,
 \qquad \int_0^{r_0}\frac{|d(r)|^2}{r}\,\mathrm dr
       \leq C\|\nabla_hE\|_0^2,\\
 h_{e_0}^{-1}\|d\|_{0,e_0}^2+|d|_{H^{1/2}(e_0)}^2
 +\int_{e_0}\frac{|d(s)|^2}{\operatorname{dist}(s,\partial e_0)}\,\mathrm ds
 \leq C\|\nabla_hE\|_0^2.
 \end{gathered}
\end{equation*}
Since $d$ is polynomial, the weighted estimate implies
$d|_{\partial e_0}=0$.
We apply \eqref{eq:p-robust-zero-edge-extension}
on one adjacent triangle and extend the result by zero to the other
triangles, with the sign chosen so that
\begin{equation*}
 \jump{b_{e_0}}|_{e_0}=d,\qquad b_{e_0}|_{\partial T\setminus e_0}=0,
 \qquad \|\nabla_h b_{e_0}\|_0\leq C\|\nabla_hE\|_0.
\end{equation*}
Define the corrected polynomial by $\widetilde z:=z-b_{e_0}$. Then
\begin{equation*}
 \jump{\widetilde z-w}|_{e_0}
 =\jump E|_{e_0}-\jump{b_{e_0}}|_{e_0}=0.
\end{equation*}
In what follows, we denote this corrected polynomial again by $z$.

To preserve the normal traces on $\Gamma_{\rm rad}$, we apply
the scalar construction starting from the prescribed trace on the left
or right radial edge, and denote these operators by
$\mathcal E_L$ and $\mathcal E_R$. At a corner, the two normals are
linearly independent. Thus the matrix
\begin{equation*}
 M:=\begin{pmatrix}n_L^\top\\n_R^\top\end{pmatrix}
\end{equation*}
is invertible, and we define
\begin{equation*}
 z:=M^{-1}\begin{pmatrix}\mathcal E_L(n_L\cdot w)\\
                         \mathcal E_R(n_R\cdot w)\end{pmatrix}.
\end{equation*}
This definition gives
\begin{equation*}
 z\cdot n_L=w\cdot n_L\quad\text{on }e_L,
 \qquad
 z\cdot n_R=w\cdot n_R\quad\text{on }e_R,
\end{equation*}
and the scalar stability estimates imply
\begin{equation*}
 \|\nabla_hz\|_0\leq C\|M^{-1}\|\|M\|\|\nabla_hw\|_0.
\end{equation*}

On a straight side, let $n$ and $t$ be a unit normal and a unit tangent,
respectively. Let $z_t$ be the scalar lift obtained by applying the
unrestricted construction to $w\cdot t$, and define
\begin{equation*}
 z_n^{(0)}:=\mathcal E_L(n\cdot w),\qquad
 E_n:=z_n^{(0)}-n\cdot w.
\end{equation*}
The prescribed outer and left traces imply
\begin{equation*}
 E_n|_{\Gamma_{\rm out}\cup e_L}=0,\qquad
 d:=E_n|_{e_R}\in\mathcal P_m(e_R).
\end{equation*}
For the edge correction $b_{e_R}$ with $b_{e_R}|_{e_R}=d$
constructed below, define
\begin{equation*}
 z_n:=z_n^{(0)}-b_{e_R},\qquad z:=z_n n+z_t t.
\end{equation*}
Since $b_{e_R}$ vanishes on the other triangle edges, this correction
preserves the outer and left traces and gives
\begin{equation*}
 z\cdot n=w\cdot n\quad\text{on }e_L\cup e_R.
\end{equation*}
The endpoint trace estimate and
\eqref{eq:p-robust-zero-edge-extension} give $b_{e_R}$,
zero on the other two triangle edges and extended by zero, with
$\|\nabla b_{e_R}\|_0\leq C\|\nabla_hE_n\|_0$.
Set $v_m:=b+z$. Since $w=v-b$, $\jump z=\jump w$, and
$z|_{\Gamma_{\rm out}}=0$, this polynomial has the prescribed jumps
and outer traces. When normal traces are retained,
$z\cdot n=w\cdot n$ on $\Gamma_{\rm rad}$ also gives
$v_m\cdot n=v\cdot n$ there. Moreover,
\eqref{eq:p-robust-partial-extension} and the bounds above give
\begin{align*}
 \|\nabla_hb\|_0&\leq C(\gamma_i)\|\nabla_hv\|_0,\\
 \|\nabla_hv_m\|_0
 &\leq\|\nabla_hb\|_0+\|\nabla_hz\|_0\\
 &\leq\|\nabla_hb\|_0+C\|\nabla_hw\|_0
 \leq C\|\nabla_hv\|_0.
\end{align*}
Here $C=C(\gamma_i)$ in the basic case and $C=C(\gamma_i,\Omega)$
when normal traces are retained.
The additional dependence on $\Omega$ comes from the factor
$\|M^{-1}\|\|M\|$ in the corner reconstruction used to preserve normal
traces; this factor depends only on the corner angles of $\Omega$.
This completes the proof.
\end{proof}

\subsection{Proof of the common lifting estimate}
\begin{proof}[Proof of Theorem~\ref{thm:p-robust-local-lifting}]
Fix $W_i\in\{X_i,Y_i\}$ and let $\psi_i\in W_i$ be the Riesz
representative of $r_i$. Applying
Lemma~\ref{lem:p-robust-interior-lifting} elementwise to
$\tau=\nabla^2\psi_i$ gives
\begin{gather*}
 \begin{gathered}
 \tau=\nabla^2\psi_i,\qquad
 (\tau,\nabla^2w)_{\Omega_i}=r_i(w),\quad w\in W_i,\qquad
 \|\tau\|_0=\|r_i\|_{W_i'},\\
 \tau=\tau^g+\symcurl_hv,\qquad
 \tau^g|_T\in\mathcal P_k(T,\mathbb S),\qquad
 \|\tau^g\|_0+\|\nabla_hv\|_0\leq C\|\tau\|_0.
 \end{gathered}
\\
\shortintertext{
Testing $\symcurl_hv=\tau-\tau^g$ against $\nabla^2w$ and using
\eqref{eq:p-robust-potential-boundary}, we obtain on
each interior edge:}
 \begin{aligned}
 -\partial_t(\jump{v}\cdot n)&=r_i^n-\jump{\tr_1\tau^g}\in\mathcal P_k(e),\\
 \partial_t^2(\jump{v}\cdot t)&=r_i^b+\jump{\tr_2\tau^g}\in\mathcal P_{k-1}(e).
 \end{aligned}
\\
\shortintertext{These identities imply}
 \jump{v}\in\mathcal P_{k+1}(e)^2.
\\
\shortintertext{
On $\Gamma_{\rm out}$, the same argument with one-sided
traces and $r_i^n=r_i^b=0$ gives $v|_e\in\mathcal P_{k+1}(e)^2$.
For $W_i=Y_i$, testing with free normal derivatives on
$\Gamma_{\rm rad}$ gives}
 \tr_1\tau=0,\qquad
 \partial_t(v\cdot n)=\tr_1\tau^g\in\mathcal P_k(e).
\\
\shortintertext{It follows that}
 v\cdot n\in\mathcal P_{k+1}(e).
\\
\shortintertext{
These polynomial jumps and boundary traces allow us to apply
Lemma~\ref{lem:p-robust-fan}, which provides $v_{k+1}$ with}
 \begin{gathered}
 \tau_{i,k}:=\tau^g+\symcurl_hv_{k+1}\in\mathcal P_k(\mathcal T_i,\mathbb S),
 \quad\|\tau_{i,k}\|_0\leq C_i\|\tau\|_0,\\
 d:=v_{k+1}-v\in H^1(\Omega_i,\mathbb R^2),\quad
 d|_{\Gamma_{\rm out}}=0,\quad
 \text{and}\quad d\cdot n|_{\Gamma_{\rm rad}}=0\quad\text{if }W_i=Y_i.
 \end{gathered}
\\
\shortintertext{
Since $d=0$ on $\Gamma_{\rm out}$, while on $\Gamma_{\rm rad}$
either $w=\partial_nw=0$ for $w\in X_i$ or
$w=d\cdot n=0$ for $w\in Y_i$, weak integration by parts gives}
 (\symcurl d,\nabla^2w)_{\Omega_i}=0,\qquad w\in W_i.
\\
\shortintertext{
Together with $\tau_{i,k}=\tau+\symcurl d$, this yields}
 (\tau_{i,k},\nabla^2w)_{\Omega_i}
 =(\tau,\nabla^2w)_{\Omega_i}=r_i(w),\qquad w\in W_i.
\\
\shortintertext{
Applying the elementwise Green formula and comparing with the residual
representation for $w\in W_i$ gives:}
 \begin{gathered}
 \dd_h\tau_{i,k}=r_i^o,\;
 \jump{\tr_1\tau_{i,k}}=r_i^n,\;
 -\jump{\tr_2\tau_{i,k}}=r_i^b,\;
 \jump{\tr_{3,a}\tau_{i,k}}=r_i^g(a),\quad a\in\mathcal V_i,\\
 \tr_1\tau_{i,k}=\tr_2\tau_{i,k}=0\text{ on }\Gamma_{\rm out},\qquad
 \jump{\tr_{3,a}\tau_{i,k}}=0,\quad a\in\mathcal V_i^c.
 \end{gathered}
\end{gather*}

Thus $\tau_{i,k}$ belongs to the corresponding reconstruction space.
Minimality in the corresponding feasible space gives the same bound
for each minimum-norm reconstruction, completing both assertions.
\end{proof}

\section{Polynomial-degree-independent residual localization}
\label{app:p-robust-localization}\label{app:poisson-p-robust}
We use the patch spaces and notation introduced in
Appendix~\ref{app:p-robust-lifting}.

\subsection{Localization for the Argyris partition}
\begin{lemma}
\label{lem:p-robust-localization}
Let $R\in H^{-2}(\Omega)$ satisfy $\langle R,\phi_i\ell\rangle=0$
for every $\ell\in\mathcal P_1(\Omega_i)$ on each interior patch $\Omega_i$.
Let $E\in H_0^2(\Omega)$ be the
Riesz representative of $R$ with respect to $(\nabla^2\cdot,\nabla^2\cdot)_\Omega$.
Then, for the degree-five Argyris partition, the localized residual $R_i$
belongs to $X_i'$ and satisfies
\begin{gather*}
 \|R_i\|_{X_i'}\leq C(\gamma_i)|E|_{2,\Omega_i}.
\end{gather*}
Consequently,
\begin{equation*}
 \sum_i\|R_i\|_{X_i'}^2\leq C(\gamma)\|R\|_{H^{-2}(\Omega)}^2.
\end{equation*}
\end{lemma}
\begin{proof}
Fix an arbitrary vertex $a_i$, its patch $\Omega_i$, and the
corresponding Argyris function $\phi_i$. We have
$\phi_i=\nabla\phi_i=0$ on $\Gamma_{\rm out}$.
Therefore, for any $v\in X_i$, zero extension gives
$\phi_iv\in H_0^2(\Omega)$.

Set $z=v-\ell_i(v)$. Since $\nabla^2\ell_i(v)=0$,
\eqref{eq:argyris-partition-bound} and \eqref{eq:p-robust-patch-poincare} give
\begin{equation*}
 |\phi_i z|_{2,\Omega_i}
 \leq C(\gamma_i)\bigl(|z|_{2,\Omega_i}+h_i^{-1}\|\nabla z\|_{0,\Omega_i}
                         +h_i^{-2}\|z\|_{0,\Omega_i}\bigr)
 \leq C(\gamma_i)|v|_{2,\Omega_i}.
\end{equation*}
Since $R_i(v)=R_i(z)$, the Riesz representation and the
Cauchy--Schwarz inequality yield
\begin{equation*}
 |R_i(v)|=|(\nabla^2E,\nabla^2(\phi_i z))_{\Omega_i}|
 \leq C(\gamma_i)|E|_{2,\Omega_i}|v|_{2,\Omega_i}.
\end{equation*}
Taking the supremum over $v\in X_i$ proves the patch bound.
Finally, each triangle belongs to three vertex patches, so
\begin{equation*}
 \sum_i\|R_i\|_{X_i'}^2\leq C(\gamma)\sum_i|E|_{2,\Omega_i}^2
 \leq 3C(\gamma)|E|_{2,\Omega}^2=3C(\gamma)\|R\|_{H^{-2}(\Omega)}^2.
\end{equation*}
\end{proof}

\subsection{Localization for the linear Poisson partition}
Under the hypotheses of Theorem~\ref{thm:poisson-p-estimator},
set $e=u-u_h$ and $P=p+1$. The residual satisfies
\begin{gather*}
 \begin{gathered}
 G_T:=f+\Delta_hu_h\in\mathcal P_{p-2}(T),\qquad
 J_F:=-\jump{\partial_nu_h}\in\mathcal P_{p-1}(F),\\
 \langle R,v\rangle=(\nabla e,\nabla v)
  =\sum_T(G_T,v)_T+\sum_{F\in\mathring{\mathcal E}_h}(J_F,v)_F
       ,\qquad v\in H_0^1(\Omega).
 \end{gathered}
\end{gather*}
From now on, $\phi_i$ denotes the hat function at the node $a_i$.
For any $w\in Y_i$, extending the product $\phi_iw$ by zero to $\Omega$
gives $\phi_iw\in H_0^1(\Omega)$.
We also define $R_i\in Y_i'$ by $R_i(w):=\langle R,\phi_iw\rangle$.
Then $R_i(\ell)=0$ holds for all $\ell\in\mathcal L_i$.

\begin{lemma}
\label{lem:poisson-p-error-trace}\label{lem:poisson-p-residual}
Let $\omega_F$ be the union of the two triangles sharing an interior edge
$F\in\mathcal E_i$. Then we have
\begin{equation}\label{eq:poisson-local-residual-bounds}
 \begin{aligned}
 h_T^2\|G_T\|_{0,T}&\leq C(\gamma_i)P^4\|e\|_{0,T},
 \qquad\forall T\in\mathcal T_i,\\
 h_F^{3/2}\|J_F\|_{0,F}&\leq C(\gamma_i)P^3\|e\|_{0,\omega_F},
 \qquad\forall F\in\mathcal E_i.
 \end{aligned}
\end{equation}
\end{lemma}
\begin{proof}
Choose $\chi_T\in C_c^\infty(T)$ with $0\leq\chi_T\leq1$ such that
\begin{equation*}
 \chi_T(x)=
 \begin{cases}
 0,&\operatorname{dist}(x,\partial T)\leq\delta_T/2,\\
 1,&\operatorname{dist}(x,\partial T)\geq\delta_T.
 \end{cases}
\end{equation*}
Here $\delta_T=ch_TP^{-2}$, where $c>0$ is a sufficiently small constant
to be determined. We can choose $\chi_T$ so that
$\|D^j\chi_T\|_{L^\infty(T)}\leq C(\gamma_i)\delta_T^{-j}$ for $j=1,2$.

Let $E\subset\partial T$ be an edge and $a$ its opposite vertex.
Let $\lambda_a$ be the barycentric coordinate associated with $a$.
Denote by $H_E$ the altitude of $T$ relative to $E$.
For $0\leq t\leq\delta_T\leq H_E/2$, set
\begin{equation*}
 E_t:=\{\lambda_a=t/H_E\},\qquad
 T_t:=\{\lambda_a>t/H_E\}.
\end{equation*}
Then
\begin{equation*}
 \int_{E_t}G_T^2\,ds
 \leq CP^2\frac{|E_t|}{|T_t|}\|G_T\|_{0,T_t}^2
 \leq C(\gamma_i)P^2h_T^{-1}\|G_T\|_{0,T}^2,
\end{equation*}
where we use the inverse trace inequality
\citep[Theorem~3]{WarburtonHesthaven2003} and the inclusion $T_t\subseteq T$.
Integrating with respect to $t$ and summing over the three edges of $T$ gives
{\small\thickmuskip=2mu\thinmuskip=2mu
\begin{equation*}
 \int_{\{\operatorname{dist}(x,\partial T)<\delta_T\}}G_T^2\,dx
 \leq\sum_{E\subset\partial T}\int_0^{\delta_T}\int_{E_t}G_T^2\,ds\,dt
 \leq C(\gamma_i)P^2\delta_Th_T^{-1}\|G_T\|_{0,T}^2
 =C(\gamma_i)c\|G_T\|_{0,T}^2,
\end{equation*}
}%
\noindent where the small constant $c$ can be chosen such that
$\delta_T\leq\frac12\min_{E\subset\partial T}H_E$ and
$C(\gamma_i)c\leq\frac12$.
Now set $v_T:=\chi_TG_T\in H_0^2(T)$. The inverse inequality
\citep[Theorem~4.76]{Schwab1998} gives that
\begin{equation*}
 \|\Delta v_T\|_{0,T}
 \leq C(\gamma_i)\left(P^4h_T^{-2}
 +\delta_T^{-1}P^2h_T^{-1}+\delta_T^{-2}\right)\|G_T\|_{0,T}
 \leq C(\gamma_i)P^4h_T^{-2}\|G_T\|_{0,T}.
\end{equation*}
Therefore, we have that
\begin{equation*}
 (G_T,v_T)_T=-(e,\Delta v_T)_T
 \leq C(\gamma_i)P^4h_T^{-2}\|e\|_{0,T}\|G_T\|_{0,T}.
\end{equation*}
On the other hand, we also have
\begin{equation*}
 (G_T,v_T)_T=\int_T\chi_TG_T^2\,dx
 \geq\|G_T\|_{0,T}^2
 -\int_{\{\operatorname{dist}(x,\partial T)<\delta_T\}}G_T^2\,dx
 \geq\tfrac12\|G_T\|_{0,T}^2.
\end{equation*}
Combining these estimates proves the first inequality in
\eqref{eq:poisson-local-residual-bounds}.

Let $s\in(0,h_F)$ be the arclength coordinate on $F$ from one endpoint
and $r$ the signed normal coordinate. Choose
$\chi_F\in C_c^\infty(0,h_F)$ with $0\leq\chi_F\leq1$ such that
\begin{equation*}
 \chi_F(s)=
 \begin{cases}
 0,&\min\{s,h_F-s\}\leq\delta_F/2,\\
 1,&\min\{s,h_F-s\}\geq\delta_F.
 \end{cases}
\end{equation*}
Here $\delta_F=ch_FP^{-2}$, where $c>0$ is a sufficiently small constant
to be determined. We can choose $\chi_F$ so that
$\|\chi_F^{(j)}\|_{L^\infty(F)}\leq C\delta_F^{-j}$ for $j=1,2$.
Applying the inverse trace inequality
\citep[Theorem~2]{WarburtonHesthaven2003} to $(t,h_F)$ and $(0,h_F-t)$ gives
\begin{equation*}
 |J_F(t)|^2+|J_F(h_F-t)|^2
 \leq CP^2h_F^{-1}\|J_F\|_{0,F}^2,
 \qquad 0\leq t\leq\delta_F\leq h_F/2.
\end{equation*}
Integrating with respect to $t$ gives
\begin{equation*}
 \int_{\{\min(s,h_F-s)<\delta_F\}}J_F^2\,ds
 \leq CP^2\delta_Fh_F^{-1}\|J_F\|_{0,F}^2
 =Cc\|J_F\|_{0,F}^2,
\end{equation*}
where the small constant $c$ can be chosen such that
$\delta_F\leq h_F/2$ and $Cc\leq\frac12$.

Let $\eta\in C_c^\infty(-1,1)$ be even with $\eta(0)=1$.
Choose $\kappa=\kappa(\gamma_i)>0$ small enough that
\begin{equation*}
 [\delta_F/2,h_F-\delta_F/2]\times[-\kappa\delta_F,\kappa\delta_F]
 \mathrel{\subset\!\subset}\omega_F.
\end{equation*}
Now set $v_F(s,r):=\chi_F(s)J_F(s)\eta(r/(\kappa\delta_F))
\in H_0^2(\omega_F)$. Integrating with respect to $r$ gives
\begin{equation*}
 \|v_F\|_{0,\omega_F}
 \leq C(\gamma_i)\delta_F^{1/2}\|J_F\|_{0,F}
 \leq C(\gamma_i)h_F^{1/2}P^{-1}\|J_F\|_{0,F}.
\end{equation*}
The inverse inequality \citep[Theorem~3.91]{Schwab1998} gives that
\begin{equation*}
 \begin{aligned}
 \|\Delta v_F\|_{0,\omega_F}
 &\leq C(\gamma_i)\delta_F^{1/2}
 \left(P^4h_F^{-2}+\delta_F^{-1}P^2h_F^{-1}+\delta_F^{-2}\right)
 \|J_F\|_{0,F}\\
 &\leq C(\gamma_i)P^3h_F^{-3/2}\|J_F\|_{0,F}.
 \end{aligned}
\end{equation*}
The residual identity and the first inequality in
\eqref{eq:poisson-local-residual-bounds} give that
\begin{equation*}
 \begin{aligned}
 (J_F,v_F)_F
 &=-(e,\Delta v_F)_{\omega_F}
   -\sum_{T\subset\omega_F}(G_T,v_F)_T\\
 &\leq\|e\|_{0,\omega_F}
 \left(\|\Delta v_F\|_{0,\omega_F}
 +C(\gamma_i)P^4h_F^{-2}\|v_F\|_{0,\omega_F}\right)\\
 &\leq C(\gamma_i)P^3h_F^{-3/2}
 \|e\|_{0,\omega_F}\|J_F\|_{0,F}.
 \end{aligned}
\end{equation*}
On the other hand, we also have
\begin{equation*}
 (J_F,v_F)_F=\int_F\chi_FJ_F^2\,ds
 \geq\|J_F\|_{0,F}^2
 -\int_{\{\min(s,h_F-s)<\delta_F\}}J_F^2\,ds
 \geq\tfrac12\|J_F\|_{0,F}^2.
\end{equation*}
Combining these estimates proves the second inequality in
\eqref{eq:poisson-local-residual-bounds}.
\end{proof}

\begin{lemma}\label{lem:poisson-p-patch-polynomial}
For $w\in Y_i$ and $z=w-\ell_i(w)$ satisfying
\eqref{eq:p-robust-patch-poincare}, there exists a single polynomial
$q\in\mathcal P_{p-1}(\mathbb R^2)|_{\Omega_i}$ satisfying
\begin{gather*}
 \begin{gathered}
 q|_{\Gamma_{\rm rad}}=0,\qquad
 \|D^j(z-q)\|_{0,\Omega_i}\leq Ch_i^{2-j}P^{j-2}|w|_{2,\Omega_i}
       ,\quad j=0,1,\\
 |q|_{2,\Omega_i}+h_i^{-1}\|\nabla q\|_{0,\Omega_i}\leq C|w|_{2,\Omega_i},
 \qquad
 \|z-q\|_{0,F}\leq Ch_i^{3/2}P^{-3/2}|w|_{2,\Omega_i}
 \end{gathered}
\end{gather*}
for every edge $F$ of $\mathcal T_i$.
\end{lemma}
\begin{proof}
Set $F_i(\widehat x)=a_i+h_i\widehat x$,
$\widehat\Omega_i=F_i^{-1}(\Omega_i)$,
$\widehat z=z\circ F_i$, and $\widehat w=w\circ F_i$.
We work on the rescaled patch of unit diameter with distinguished vertex at the origin, omitting hats below.

On straight boundary patches, rotate the physical side to $x_1=0$;
at corners, linearly map the two sides to the coordinate axes, retaining notation.
Let $\Omega_i^{\rm odd}$ be the open patch obtained by joining $\Omega_i$
and its mirror copies across these axes.
Define $z_{\rm odd}$ there by retaining $z$ on $\Omega_i$, changing sign under each
reflection, and taking zero on the axes corresponding to the physical sides.
Since $z$ has zero boundary values on $\Gamma_{\rm rad}$ in the Sobolev sense,
its odd reflection has matching values and first derivatives across the interfaces;
hence $z_{\rm odd}\in H^2(\Omega_i^{\rm odd})$.
On interior patches, set $\Omega_i^{\rm odd}=\Omega_i$ and $z_{\rm odd}=z$.
The bounded extension theorem
\citep[Theorem~5.24]{AdamsFournier2003}
provides an extension $U_0\in H^2(\mathbb R^2)$
of $z_{\rm odd}$ such that
\begin{gather*}
 U_0|_{\Omega_i^{\rm odd}}=z_{\rm odd},\\
 \|U_0\|_{H^2(\mathbb R^2)}
 \leq C\|z_{\rm odd}\|_{H^2(\Omega_i^{\rm odd})}
 \leq C\|z\|_{H^2(\Omega_i)}\leq C|w|_{2,\Omega_i}.
\\
\shortintertext{To make this extension odd on $\mathbb R^2$, define the following projection.
Set $S_1(x_1,x_2):=(-x_1,x_2)$ and $S_2(x_1,x_2):=(x_1,-x_2)$, and define}
 \Pi v:=\begin{cases}
 v,&\text{interior patch},\\
 \tfrac12(v-v\circ S_1),&\text{straight boundary patch},\\
 \tfrac14(v-v\circ S_1-v\circ S_2+v\circ S_1\circ S_2),
 &\text{corner patch}.
 \end{cases}
\\
\shortintertext{On any domain $D$ invariant under the relevant reflections,
it is easy to prove that}
 \|\Pi v\|_{H^j(D)}\leq\|v\|_{H^j(D)},\qquad j=0,1,2.
\\
\shortintertext{Finally, set $U:=\Pi U_0$.
Since $\Omega_i^{\rm odd}$ is invariant under the relevant reflections
and $\Pi z_{\rm odd}=z_{\rm odd}$, we have
$U|_{\Omega_i^{\rm odd}}=z_{\rm odd}$, $U|_{\Omega_i}=z$,
and $\Pi U=U$, with}
 \|U\|_{H^2(\mathbb R^2)}\leq\|U_0\|_{H^2(\mathbb R^2)}
 \leq C|w|_{2,\Omega_i}.
\\
\shortintertext{For $p\geq3$, set $m=\lfloor(p-1)/2\rfloor\geq1$ and choose a fixed square
$I\supset\Omega_i$ centered at the origin. By
\citet[extended version, Theorem~A.3]{Melenk2005}, there is $r_m\in\mathcal Q_m(I)$,
of degree at most $m$ in each coordinate, such that}
 \|U-r_m\|_{H^j(I)}
 \leq C(m+1)^{j-2}|U|_{2,I},\qquad j=0,1,2.
\\
\shortintertext{Set $\widetilde q=\Pi r_m$.
Since $\Pi U=U$ and $I$ is symmetric under these reflections,
the estimates are preserved, while
$\widetilde q\in\mathcal P_{2m}\subset\mathcal P_{p-1}$ and
$\widetilde q|_{\Gamma_{\rm rad}}=0$.
Since $m+1\eqsim P$,}
 \begin{aligned}
 \|z-\widetilde q\|_{H^j(\Omega_i)}
 &\leq \|\Pi(U-r_m)\|_{H^j(I)}
 \leq \|U-r_m\|_{H^j(I)}\\
 &\leq C(m+1)^{j-2}|U|_{2,I}
 \leq CP^{j-2}|w|_{2,\Omega_i},\qquad j=0,1,\\[2pt]
 \|\widetilde q\|_{H^2(\Omega_i)}
 &\leq \|U\|_{H^2(I)}+\|\Pi(U-r_m)\|_{H^2(I)}\\
 &\leq \|U\|_{H^2(I)}+C|U|_{2,I}
 \leq C|w|_{2,\Omega_i}.
 \end{aligned}
\\
\shortintertext{Pulling $\widetilde q$ back defines $q$ and gives the volume and stability bounds.
For $p=2$, take $q=0$ and use \eqref{eq:p-robust-patch-poincare}.
The multiplicative trace inequality
\citep[Lemma~12.15, Eq.~(12.16)]{ErnGuermond2021} yields}
 \begin{aligned}
 \|z-q\|_{0,F}^2
 &\leq C\bigl(h_i^{-1}\|z-q\|_{0,\Omega_i}^2
       +\|z-q\|_{0,\Omega_i}\|\nabla(z-q)\|_{0,\Omega_i}\bigr)\\
 &\leq Ch_i^3(P^{-4}+P^{-3})|w|_{2,\Omega_i}^2.
 \end{aligned}
\end{gather*}
\end{proof}

The zero extension of $\phi_iq$ to $\Omega$ lies in $V_h$,
since it is continuous and piecewise of degree at most $p$,
with $\phi_i|_{\Gamma_{\rm out}}=0$ and $q|_{\Gamma_{\rm rad}}=0$.

\begingroup
\begin{lemma}\label{lem:poisson-p-smoothing}
Let $v\in H_0^1(\Omega_i)$ with $v|_T\in H^2(T)$ for
$T\in\mathcal T_i$, and set
\begin{equation*}
 A(v)^2:=\sum_{T\in\mathcal T_i}|v|_{2,T}^2,\qquad
 B(v)^2:=\sum_{F\in\mathcal F_i}\|\jump{\partial_n v}\|_{0,F}^2,
\end{equation*}
where $\mathcal F_i:=\mathcal E_i\cup\{F\subset\Gamma_{\rm out}\}$.
For $0<\delta\leq c_0h_i$, there exists
$\mathcal S_\delta v\in H^2(\Omega_i)$ such that
\begin{equation*}
 \mathcal S_\delta v=\partial_n\mathcal S_\delta v=0
       \text{ on }\Gamma_{\rm out},\qquad
 \mathcal S_\delta v=0\text{ on }\Gamma_{\rm rad},
\end{equation*}
with the following bounds
\begin{equation}\label{eq:poisson-p-smoothing}
 \begin{aligned}
 |\mathcal S_\delta v|_{2,\Omega_i}
 &\leq C\bigl(A(v)+\delta^{-1/2}B(v)\bigr),\\
 \|v-\mathcal S_\delta v\|_{0,\Omega_i}
 &\leq C\bigl(\delta^2 A(v)+\delta^{3/2}B(v)\bigr),\\
 \|\nabla(v-\mathcal S_\delta v)\|_{0,\Omega_i}
 &\leq C\bigl(\delta A(v)+\delta^{1/2}B(v)\bigr),\\
 \left(\sum_{F\in\mathcal E_i}
       \|v-\mathcal S_\delta v\|_{0,F}^2\right)^{1/2}
 &\leq C\bigl(\delta^{3/2}A(v)+\delta B(v)\bigr).
 \end{aligned}
\end{equation}
The constants $c_0>0$ and $C$ depend only on $\gamma_i$ and,
for a boundary corner, on its fixed angle. The zero extension to
$\Omega$ belongs to $\widetilde H^2(\Omega)$.
\end{lemma}
\begin{proof}
First consider an interior patch. Extend $v$ by zero to
$\widetilde v\in H^1(\mathbb R^2)$.
Write $\|\cdot\|_0:=\|\cdot\|_{L^2(\mathbb R^2)}$.
Integration by parts gives
\begin{equation*}
 \langle\nabla^2\widetilde v,\tau\rangle_{\mathbb R^2}
 =\sum_{T\in\mathcal T_i}\bigl(\nabla^2(v|_T),\tau\bigr)_T
 -\sum_{F\in\mathcal F_i}\bigl(\jump{\nabla\widetilde v}_F\otimes n_F,\tau\bigr)_F,
 \quad\forall\tau\in H^1(\mathbb R^2;\mathbb S),
\end{equation*}
with $\nabla^2\widetilde v\in H^{-1}(\mathbb R^2;\mathbb S)$.
Following the proof of \citet[preprint version, Lemma~2.4]{Zhang2017DiscreteKorn},
we obtain, for $y\in\mathbb R^2$,
\begin{equation}\label{eq:poisson-p-gradient-difference}
 \|\nabla(\widetilde v(\cdot-y)-\widetilde v)\|_0^2
 \leq C(\gamma_i)\bigl(|y|^2A(v)^2+|y|B(v)^2\bigr).
\end{equation}

Choose a nonnegative radial kernel $\rho\in C_c^\infty(B(0,1))$
with $\int\rho=1$, and set
$\rho_\delta(x):=\delta^{-2}\rho(x/\delta)$.
Let $b\in L^2(F)$ and extend $b_0(q):=b(z_0+qt_F)$ by zero to
$\mathbb R$, where $z_0$ is an endpoint of $F$. We have
\begin{equation*}
 \begin{aligned}
 &\left\|\int_F\rho_\delta(\cdot-z)b(z)\,\mathrm ds_z\right\|_0^2
 =\int_{\mathbb R^2}\left|\int_{\mathbb R}
 \rho_\delta((s-q)t_F+rn_F)b_0(q)\,\mathrm dq\right|^2\,\mathrm ds\,\mathrm dr\\
 &\leq\delta^{-1}\|b\|_{0,F}^2
 \int_{\mathbb R}\left(\int_{\mathbb R}
 |\rho(\sigma t_F+\eta n_F)|\,\mathrm d\sigma\right)^2\,\mathrm d\eta
 \leq C\delta^{-1}\|b\|_{0,F}^2,
 \end{aligned}
\end{equation*}
where Young's inequality is used in the first inequality.
Set $M_\delta v:=\rho_\delta*\widetilde v$.
Since $\int_{\mathbb R^2}\nabla\rho_\delta(y)\,\mathrm dy=0$,
Minkowski's inequality and \eqref{eq:poisson-p-gradient-difference} give
\begin{equation*}
 \begin{aligned}
 &|M_\delta v|_{2,\mathbb R^2}
 =\left\|\int_{\mathbb R^2}
 \nabla\bigl(\widetilde v(\cdot-y)-\widetilde v\bigr)
 \otimes\nabla\rho_\delta(y)\,\mathrm dy\right\|_0\\
 &\leq C(\gamma_i)\int_{\mathbb R^2}|\nabla\rho_\delta(y)|
 \bigl(|y|A(v)+|y|^{1/2}B(v)\bigr)\,\mathrm dy
 \leq C(\gamma_i)\bigl(A(v)+\delta^{-1/2}B(v)\bigr),
 \end{aligned}
\end{equation*}
and similarly,
\begin{equation*}
 \|\nabla(M_\delta v-\widetilde v)\|_0
 \leq\int\rho_\delta(y)
 \|\nabla(\widetilde v(\cdot-y)-\widetilde v)\|_0\,\mathrm dy
 \leq C(\gamma_i)\bigl(\delta A(v)+\delta^{1/2}B(v)\bigr).
\end{equation*}
Since $\int y\rho_\delta(y)\,\mathrm dy=0$, integration along segments gives
\begin{equation*}
 M_\delta v-\widetilde v
 =\int\rho_\delta(y)\int_0^1
 y\cdot\nabla\bigl(\widetilde v-\widetilde v(\cdot-ty)\bigr)
 \,\mathrm dt\,\mathrm dy.
\end{equation*}
Applying \eqref{eq:poisson-p-gradient-difference} again yields
\begin{equation*}
 \|M_\delta v-\widetilde v\|_0
 \leq C(\gamma_i)\bigl(\delta^2A(v)+\delta^{3/2}B(v)\bigr).
\end{equation*}

Set $U_r:=\{x\in\Omega_i:\operatorname{dist}(x,\Gamma_{\rm out})<r\}$.
Choose $K=K(\gamma_i)>1$ so that
$E_x:=B(x,Kr)\setminus\overline{\Omega_i}$ satisfies
$|E_x|\geq c(\gamma_i)r^2$ for $x\in U_r$ and
$0<r<c(\gamma_i)h_i$.
Since $\nabla\widetilde v=0$ on $E_x$, also using
\eqref{eq:poisson-p-gradient-difference}, we obtain
\begin{equation*}
 \begin{aligned}
 &\|\nabla v\|_{0,U_r}^2
 =\int_{U_r}\frac1{|E_x|}\int_{E_x}
 \bigl|\nabla\widetilde v(x)-\nabla\widetilde v(z)\bigr|^2
 \,\mathrm dz\,\mathrm dx\\
 &\leq C(\gamma_i)r^{-2}\int_{|y|<Kr}\int_{U_r}
 \bigl|\nabla\widetilde v(x)-\nabla\widetilde v(x-y)\bigr|^2
 \,\mathrm dx\,\mathrm dy
 \leq C(\gamma_i)\bigl(r^2A(v)^2+rB(v)^2\bigr).
 \end{aligned}
\end{equation*}
The zero-trace Poincar\'e inequality on boundary balls then gives
\begin{equation*}
 \|v\|_{0,U_r}\leq C(\gamma_i)r\|\nabla v\|_{0,U_{4r}}
 \leq C(\gamma_i)\bigl(r^2A(v)+r^{3/2}B(v)\bigr).
\end{equation*}
Taking $r=4\delta$ and combining these bounds with the convolution
errors, we obtain
\begin{equation*}
 \|M_\delta v\|_{0,U_{4\delta}}
 +\delta\|\nabla M_\delta v\|_{0,U_{4\delta}}
 \leq C(\gamma_i)\bigl(\delta^2A(v)+\delta^{3/2}B(v)\bigr).
\end{equation*}

Choose a cutoff $\chi_\delta\in C_c^\infty(\Omega_i)$ with $0\leq\chi_\delta\leq1$ such that
\begin{equation*}
 \chi_\delta(x)=
 \begin{cases}
 0,&\operatorname{dist}(x,\Gamma_{\rm out})\leq2\delta,\\
 1,&\operatorname{dist}(x,\Gamma_{\rm out})\geq4\delta.
 \end{cases}
\end{equation*}
We can choose $\chi_\delta$ with
$\|D^j\chi_\delta\|_\infty\leq C\delta^{-j}$ for $j=1,2$.
Now set $\mathcal S_\delta v:=\chi_\delta M_\delta v$.
Then $\mathcal S_\delta v=\partial_n\mathcal S_\delta v=0$ on
$\Gamma_{\rm out}$, and
\begin{equation*}
 \begin{aligned}
 |\mathcal S_\delta v|_{2,\Omega_i}
 &\leq C\bigl(|M_\delta v|_{2,\Omega_i}
       +\delta^{-1}\|\nabla M_\delta v\|_{0,U_{4\delta}}
       +\delta^{-2}\|M_\delta v\|_{0,U_{4\delta}}\bigr)\\
 &\leq C(\gamma_i)\bigl(A(v)+\delta^{-1/2}B(v)\bigr).
 \end{aligned}
\end{equation*}
We also have
\begin{gather*}
 \|v-\mathcal S_\delta v\|_{0,\Omega_i}
 \leq\|v-M_\delta v\|_{0,\Omega_i}
       +\|M_\delta v\|_{0,U_{4\delta}}
 \leq C(\gamma_i)\bigl(\delta^2A(v)+\delta^{3/2}B(v)\bigr),\\
 \begin{aligned}
 \|\nabla(v-\mathcal S_\delta v)\|_{0,\Omega_i}
 &\leq\|\nabla(v-M_\delta v)\|_{0,\Omega_i}
       +\|\nabla M_\delta v\|_{0,U_{4\delta}}
       +C\delta^{-1}\|M_\delta v\|_{0,U_{4\delta}}\\
 &\leq C(\gamma_i)\bigl(\delta A(v)+\delta^{1/2}B(v)\bigr).
 \end{aligned}
\end{gather*}
Set $d:=v-\mathcal S_\delta v$.
The trace inequality and Young's inequality give
\begin{equation*}
 \sum_{F\in\mathcal E_i}\|d\|_{0,F}^2
 \leq C(\gamma_i)\bigl(\delta^{-1}\|d\|_{0,\Omega_i}^2
                    +\delta\|\nabla d\|_{0,\Omega_i}^2\bigr)
 \leq C(\gamma_i)\bigl(\delta^3A(v)^2+\delta^2B(v)^2\bigr).
\end{equation*}
Taking square roots proves the last bound.
We choose $c_0=c_0(\gamma_i)>0$ small enough that the boundary-ball
and trace estimates apply for $0<\delta\leq c_0h_i$.

For boundary patches, use the coordinates and notation from
the proof of Lemma~\ref{lem:poisson-p-patch-polynomial}, with hats for
normalized quantities. Let $v_{\rm odd}$ be the odd reflection of $\widehat v$.
Then $\jump{\nabla v_{\rm odd}}_F=0$ on each reflection edge $F$, and
$v_{\rm odd}\in H_0^1(\Omega_i^{\rm odd})$ is piecewise $H^2$.
With $A$ and $B$ computed on the reflected triangulation, we have
\begin{equation*}
 \begin{aligned}
 A(v_{\rm odd})^2
 &\leq4\sum_{\widehat T\in\widehat{\mathcal T}_i}
 |\widehat v|_{2,\widehat T}^2
 \leq C(\gamma_i,\Omega)h_i^2A(v)^2,\\
 B(v_{\rm odd})^2
 &\leq4\sum_{\widehat F\in\widehat{\mathcal F}_i}
 \|\jump{\partial_n\widehat v}\|_{0,\widehat F}^2
 \leq C(\gamma_i,\Omega)h_iB(v)^2.
 \end{aligned}
\end{equation*}
Since $\rho$ is radial, $M_{\delta/h_i}v_{\rm odd}=0$ on
$\widehat\Gamma_{\rm rad}$.
Applying the interior construction at scale $\delta/h_i$ and returning
to the original coordinates proves the result.
\end{proof}
\endgroup

\begin{theorem}
\label{thm:poisson-p-localization}
Let $u_h$ be the conforming degree-$p$ Lagrange FEM solution of
$-\Delta u=f$ on a convex domain $\Omega$, with $p\ge2$ and
$f|_T\in\mathcal P_{p-2}(T)$ for every $T\in\mathcal T_h$.
Then the localized residual $R_i$ satisfies
\begin{gather*}
 \|R_i\|_{Y_i'}\leq C(\gamma_i,\Omega)\|u-u_h\|_{0,\Omega_i}.
\end{gather*}
\end{theorem}
\begin{proof}
\begingroup

Fix $w\in Y_i$, set $z=w-\ell_i(w)$, and choose $q$ from
Lemma~\ref{lem:poisson-p-patch-polynomial}. Write
$r=z-q$ and $v=\phi_i r$.
Galerkin orthogonality gives $R_i(w)=\langle R,v\rangle$.
Since $r=0$ on $\Gamma_{\rm rad}$ and $\phi_i=0$ on
$\Gamma_{\rm out}$, we have $v\in H_0^1(\Omega_i)$.
By Lemma~\ref{lem:poisson-p-patch-polynomial}, we have
\begin{align*}
 A(v)&\leq C\bigl(|r|_{2,\Omega_i}
                    +h_i^{-1}\|\nabla r\|_{0,\Omega_i}\bigr)
          \leq C|w|_{2,\Omega_i},\\
 B(v)&\leq Ch_i^{-1}
 \left(\sum_{F\in\mathcal E_i\,\cup\,\{F\subset\Gamma_{\rm out}\}}
                  \|r\|_{0,F}^2\right)^{1/2}
          \leq Ch_i^{1/2}P^{-3/2}|w|_{2,\Omega_i}.
\end{align*}
Apply Lemma~\ref{lem:poisson-p-smoothing} with
$\delta=c_0h_iP^{-2}$, and set
$\widehat v=\mathcal S_\delta v$, $s=v-\widehat v$.
Equation~\eqref{eq:poisson-p-smoothing}~gives
\begin{equation*}
 |\widehat v|_{2,\Omega_i}\leq C|w|_{2,\Omega_i},\quad
 \|s\|_{0,\Omega_i}\leq Ch_i^2P^{-4}|w|_{2,\Omega_i},\quad
 \left(\sum_{F\in\mathcal E_i}\|s\|_{0,F}^2\right)^{1/2}
 \leq Ch_i^{3/2}P^{-3}|w|_{2,\Omega_i}.
\end{equation*}
The zero extension of $\widehat v$ belongs to
$\widetilde H^2(\Omega)$. Integration by parts, using
$e\in H_0^1(\Omega)$, therefore gives
\begin{equation*}
 |\langle R,\widehat v\rangle|
 =|(e,\Delta\widehat v)_{\Omega_i}|
 \leq\sqrt2\|e\|_{0,\Omega_i}|\widehat v|_{2,\Omega_i}
 \leq C\|e\|_{0,\Omega_i}|w|_{2,\Omega_i}.
\end{equation*}
For $s\in H_0^1(\Omega_i)$, use the residual coefficients:
\begin{equation*}
 \langle R,s\rangle
 =\sum_{T\in\mathcal T_i}(G_T,s)_T
      +\sum_{F\in\mathcal E_i}(J_F,s)_F.
\end{equation*}

Since $s=0$ on $\Gamma_{\rm out}$ and
$\omega_F\subseteq\Omega_i$ for $F\in\mathcal E_i$, the residual
bounds \eqref{eq:poisson-local-residual-bounds}, patchwise scale
comparability, and finite overlap give
\begin{equation*}
 |\langle R,s\rangle|
 \leq C\|e\|_{0,\Omega_i}\left[
 P^4h_i^{-2}\|s\|_{0,\Omega_i}
 +P^3h_i^{-3/2}
       \left(\sum_{F\in\mathcal E_i}\|s\|_{0,F}^2\right)^{1/2}\right]
 \leq C\|e\|_{0,\Omega_i}|w|_{2,\Omega_i}.
\end{equation*}
Combining the two parts of $R_i(w)$ and taking the supremum over
$w\in Y_i$ proves the patch estimate, with $C=C(\gamma_i,\Omega)$
independent of $h_i$ and $p$.
\endgroup
\end{proof}
\endgroup

% References precede the appendices in the IMA style.

\end{document}